\documentclass[12pt,leqno]{amsart}
\usepackage{amsmath,epsfig,graphicx,color}
\usepackage{mathrsfs,amssymb, amsfonts, latexsym, mathtools}
\usepackage{enumerate}
\usepackage{colortbl,multirow}
\definecolor{darkblue}{rgb}{.2, 0.2,.8}
\definecolor{carageen}{rgb}{0,0.5,0.3}
\definecolor{darkred}{rgb}{.8, .1,.1}

\newtheorem{lemma}{Lemma}[section]

\newtheorem{theorem}[lemma]{Theorem}

\newtheorem{proposition}[lemma]{Proposition}
\newtheorem{definition}[lemma]{Definition}
\newtheorem{corollary}[lemma]{Corollary}
\newtheorem{example}[lemma]{Example}
\newtheorem{exercise}[lemma]{Exercise}
\newtheorem{remark}[lemma]{Remark}
\newtheorem{fig}[lemma]{Figure}
\newtheorem{tab}[lemma]{Table}

\newcommand{\bth}{\begin{theorem}}
\newcommand{\ethe}{\end{theorem}}

\newcommand{\bre}{\begin{remark}\em }
\newcommand{\ere}{\end{remark}}

\newcommand{\ble}{\begin{lemma}}
\newcommand{\ele}{\end{lemma}}

\newcommand{\bde}{\begin{definition}}
\newcommand{\ede}{\end{definition}}
\newcommand{\bco}{\begin{corollary}}
\newcommand{\eco}{\end{corollary}}

\newcommand{\bpr}{\begin{proposition}}
\newcommand{\epr}{\end{proposition}}

\newcommand{\bexer}{\begin{exercise}}
\newcommand{\eexer}{\end{exercise}}

\newcommand{\bexam}{\begin{example}}
\newcommand{\eexam}{\end{example}}

\newcommand{\bfi}{\begin{fig}}
\newcommand{\efi}{\end{fig}}

\newcommand{\btab}{\begin{tab}}
\newcommand{\etab}{\end{tab}}

\newcommand{\beao}{\begin{eqnarray*}}
\newcommand{\eeao}{\end{eqnarray*}\noindent}

\newcommand{\beam}{\begin{eqnarray}}
\newcommand{\eeam}{\end{eqnarray}\noindent}

\newcommand{\beqq}{\begin{equation}}
\newcommand{\eeqq}{\end{equation}\noindent}

\newcommand{\bce}{\begin{center}}
\newcommand{\ece}{\end{center}}

\newcommand{\barr}{\begin{array}}
\newcommand{\earr}{\end{array}}

\newcommand{\vague}{\stackrel{\lower0.2ex\hbox{$\scriptscriptstyle
                    \it{v} $}}{\rightarrow}}
\newcommand{\weak}{\stackrel{\lower0.2ex\hbox{$\scriptscriptstyle
                    \it{w} $}}{\rightarrow}}
\newcommand{\what}{\stackrel{\lower0.2ex\hbox{$\scriptscriptstyle
                    \it{\hat{w}} $}}{\rightarrow}}

\newcommand{\bdis}{\begin{displaymath}}
\newcommand{\edis}{\end{displaymath}\noindent}

\newcommand{\R}{\mathbb{R}}

\newcommand{\ov}{\overline}

\newcommand{\wh}{\widehat}
\newcommand{\vep}{\varepsilon}

\newcommand{\bbr}{{\mathbb R}}

\newcommand{\cals}{{\mathcal S}}
\newcommand{\call}{{\mathcal L}}
\newcommand{\cala}{{\mathcal A}}
\newcommand{\cald}{{\mathcal D}}

\newcommand{\idrp}{{\rm ID}(\R_+)}
\newcommand{\id}{{\rm ID}(\R)}

\newcommand{\E }{{\mathbb E}}
\newcommand{\Z }{{\mathbb Z}}
\renewcommand{\P }{{\mathbb P}}

\allowdisplaybreaks

\begin{document}
\today
\bibliographystyle{plain}
\title[Local subexponentiality and infinitely divisible distributions]{Local subexponentiality and infinitely divisible distributions}
\thanks{Muneya Matsui's research is partly supported by the JSPS Grant-in-Aid for Scientific Research C
(19K11868). 
}
\author[M. Matsui]{Muneya Matsui}
\address{Department of Business Administration, Nanzan University, 18
Yamazato-cho, Showa-ku, Nagoya 466-8673, Japan.}
\email{mmuneya@gmail.com}
\author[T. Watanabe]{Toshiro Watanabe}
\address{Center for Mathematical Sciences, The University of Aizu, 
Ikkimachi Tsuruga, Aizu-Wakamatsu, Fukushima, 965-8580, Japan.
}
\email{markov2000t@yahoo.co.jp}

\begin{abstract}
We completely characterize $\Delta$- and local subexponentialities of positive-half 
compound Poisson distributions and extend the characterization on two-sided distributions. Moreover, $\Delta$-subexponentiality of 
infinitely divisible distributions is characterized with new conditions, and local subexponentiality is newly characterized 
in the two-sided case. In the process closedness properties of these subexponentialities are derived, particularly for 
distributions on $\R$. Most results are obtained by 
exploiting monotonic-type assumptions. We apply our results to 
distributions of supremum of a random work and a randomly stopped iid sum.
\end{abstract}
\keywords{infinitely divisible, long-tailedness, $\Delta$-subexponentiality, local subexponentiality, asymptotic to a non-increasing function, almost decreasing}
\subjclass[2010]{60E07, 60G70, 62F12}
\maketitle

\section{Introduction}
Studies on the subexponentiality 
of infinitely divisible distributions (IDs for abbreviation and ID in the singular form) have been initiated by Embrecht et al. 
\cite{Embrechts:Goldie:Veraverbeke:1979}, where the subexponentiality of one-sided distributions was 
completely characterized. % though that of compound Poisson distributions. 
Pakes \cite{Pakes:2004} extended the result into distributions on the real line.
Here the ``characterization of subexponentiality'' means 
the equivalence of three assertions: 
 $(\mathrm{i})$ the truncated L\'evy measure is subexponential, $(\mathrm{ii})$ the corresponding
ID is subexponential, and $(\mathrm{iii})$ L\'evy measure is long-tailed
and its tail is asymptotically equivalent to that of the corresponding ID.
Since then, under various assumptions and 
with several different versions, % of subexponentiality, 
subexponential properties of IDs
have been investigated for forty years. However, not a few issues remain unsolved   
(see the detailed explanation below). 

Recently, two monotonic-type conditions: asymptotic to a non-increasing function (a.n.i. for short) and almost decreasing 
(al.d. for short) are
 found to be useful for characterizing subexponential densities of IDs, particularly in the two-sided case, 
see Matsui \cite{matsui:2022}. The latter assumption is known to be crucial for analyzing two-sided 
subexponential densities, see Foss et al. \cite[Section 4.3]{Foss:Korshunov:Zachary:2013} and 
Finkelshtein and Tkachov 
\cite{Finkelshtein:Tkachov:2018}.

In this paper, we will exploit the two conditions in different versions of subexponentiality and 
extend the boundaries of the subexponential characterization on IDs significantly. Firstly 
we adopt four different definitions of subexponentiality and characterize their relations (see Section \ref{sec:four:relations}). 
In particular  we show that these definitions are equivalent under the a.n.i. condition.

Secondly, we focus on $\Delta$- and local subexponentialities and derive properties such as 
the closedness in convolution, asymptotic equivalence, factorization and convolution root. 
The closedness property has been one of central problems in the literature 
(see Section \ref{section:prop:sdelta:sloc} together with references their in).
$\Delta$-subexponentiality has been introduced by Asmussen et al. \cite{Asmussen:Foss:Korshunov:2003} and is 
now applied in a rather wide area, while local subexponentiality by Watanabe and Yamamuro 
\cite{Watanabe:Yamamuro:2010} is of theoretical importance, i.e. it has a close inseparable relation with $\gamma$-
subexponentiality with $\gamma>0$.

Based on the obtained properties in Section \ref{section:prop:sdelta:sloc}, 
we characterize $\Delta$- and local 
subexponentialities of IDs on the whole real line. 
%in both one-sided and two-sided cases. %, and particularly on compound distributions. 
$\Delta$-subexponentiality was only partially characterized, and local subexponentiality 
was completed only on positive-half IDs and was 
%These subexponetialities were only partially characterized, or 
not investigated at all in the two-sided case (see the table and explanation below).
We completely characterize $\Delta$- and local subexponentialities on positive-half compound Poisson distributions, 
where we found that no additional assumptions are needed. Moreover, with new and more simple assumptions 
we characterize these subexponentialities on IDs on $\R$. %Note that the local subexponentiality has not been 
%studied in the two sided case. 
Specifically, we extend the boundaries of 
$\Delta$-subexponential characterization from that of the self-decomposable distribution to that of 
$s$-self-decomposable distribution (Jurek class \cite{Jurek:1985}). These main results are 
stated in Section \ref{sec:main}. %, where we exploited the two monotonic-type assumptions. 

As applications, we characterize $\Delta$-subexponentiality of  
the distribution of supremum of a random walk and that of randomly stopped iid sum in Section \ref{sec:application},
 which are closely related with classical ruin theory and queuing theory (\cite[Section 5]{Foss:Korshunov:Zachary:2013}). 

%So far, there have been many studies about subexponentiality on the classes in ID, where 
%in the literature the subexponentiality itself has several kinds of definitions.  
%For the full class of ID, these subexponentialities are often difficult to characterize 
%without additional conditions.  Thus under various coniditions and 
%under the various subclasses such as compound Poisson (CP), self-decomposable ID, or positive-half ID, 
%the properties have been investigated. 

In the remainder of this section we introduce the four definitions of subexponentiality and 
review past researches, associating with characterization of these subexponentialities on subclasses 
of IDs. Then, we explain more precisely what we obtained in our paper. 
%relating to the preceding researches. 
This gives a good perspective of the results in this paper, and 
one could see remaining open problems also. 
Afterwards, we provide other definitions and notions necessary in the main part.

Let $F,G,H$ be probability distribution functions on $\R$ and denote by $F\ast G$ the convolution of $F$ and $G$:
\[
 F\ast G(x)=\int_{-\infty}^\infty F(x-y)G(d y)
\] 
and denote by $F^{\ast n}$ the $n$th convolution with itself. The tail probability of $F$ is denoted by $\ov F(x)=1-F(x)$. 
Let $f,g,h$ be the corresponding probability density functions on $\R$ and we use the same notations for the convolution as those for distributions, e.g. 
\[
 f\ast g(x)=\int_{-\infty}^\infty f(x-y)g(y) d y \quad \text{or}\quad f^{\ast n}(x) \quad \text{for the $n$th convolution}. 
\]
Throughout the paper, let $\Delta:=(0,c]$ with $c>0$, and for any $x$ and for any non-negative integer $n$. 
%, define also $x+\Delta:=(x,x+c]$ and $n\Delta=(0,nc]$. 
For convenience we sometimes write $F(x+\Delta):=F(x,x+c]=F(x+c)-F(x)$. Moreover, 
for functions $\alpha,\beta:\R \to \R_+$, $\alpha(x) \sim \beta(x)$ 
means that $\lim_{x\to\infty}\alpha(x)/\beta(x)\to 1$.
%Several definitions of subexponentiality and related notions. 
\begin{definition} 
$(\mathrm{i})$ A non-negative measurable function $\alpha (x)$ belongs to the class $\call$ if there exists $x_0>0$ such that 
$\alpha (x)>0,\,x\ge x_0$ and for any fixed $y>0$ $\alpha (x+y)\sim \alpha (x)$. \\
$(\mathrm{ii})$ The density $f$ of $F$ is subexponential on $\R$, denoted by $\cals$, 
if $f\in \call$ and $f^{\ast 2}(x) \sim 2f(x)$. \\
$(\mathrm{iii})$ $F$ belongs to the class $\call_{\Delta}$ if $F(x+\Delta) \in \call$. \\
$(\mathrm{iv})$ $F$ belongs to the class $\cals_{\Delta}$ if $F\in \call_{\Delta}$ 
and $F^{\ast 2}(x+\Delta) \sim 2 F(x+\Delta)$. \\
$(\mathrm{v})$ $F$ belongs to the class $\call_{loc}$ if $F\in \call_{\Delta}$ for all $\Delta:=(0,c]$ with $c>0$. \\
$(\mathrm{vi})$ $F$ belongs to the class $\cals_{loc}$ if $F\in \cals_{\Delta}$ for all $\Delta:=(0,c]$ with $c>0$.\\
$(\mathrm{vii})$ $F(dx)=f(x)dx$ belongs to the class $\call_{ac}$ if $f(x)\in \call$.\\
$(\mathrm{viii})$ $F(dx)=f(x)dx$ belongs to the class $\cals_{ac}$ if $f(x)\in \cals$.
\end{definition}

\begin{table}[htb]
\label{table1}
\centering
  \caption{Past researches on subexponential characterization}
{\renewcommand\arraystretch{1.5}
\begin{tabular}{|l||c|c|c|c|c|}  \hline
          & one-sided CP & two-sided CP & one-sided ID & two-sided ID & SD \\ \hline \hline
 $\cals$ & \cellcolor[gray]{0.8}\cite{Embrechts:Goldie:Veraverbeke:1979} 
& \cellcolor[gray]{0.8} \cite{Pakes:2004} &  \cellcolor[gray]{0.8}\cite{Embrechts:Goldie:Veraverbeke:1979}  &
\cellcolor[gray]{0.8}  \cite{Pakes:2004} & $--$ \\ \hline 
  $\cals_\Delta$ &  \cite{Asmussen:Foss:Korshunov:2003,wang:cheng:wang:2005,Watanabe:Yamamuro:2009} & \cite{Watanabe:Yamamuro:2009} &  \cite{Asmussen:Foss:Korshunov:2003,wang:cheng:wang:2005}  &  \cite{Watanabe:Yamamuro:2009}  & \cellcolor[gray]{0.8}\cite{Watanabe:Yamamuro:2010} \\ \hline 
  $\cals_{loc}$ &  \cite{Watanabe:Yamamuro:2010} &  & \cellcolor[gray]{0.8} \cite{Watanabe:Yamamuro:2010}    &   &  \cellcolor[gray]{0.8}\cite{Watanabe:Yamamuro:2010}  \\\hline 
  $\cals_{ac}$ &  \cite{shimura:watanabe:2022} & \cite{matsui:2022} &  \cite{Watanabe:2020}   &   \cite{matsui:2022} & \cellcolor[gray]{0.8}\cite{Watanabe:Yamamuro:2010}  \\\hline 
 \end{tabular}}
\end{table}

%\textcolor{red}{
%CP: compound distribution, ID: infinitely divisible distribution.} 

In the literature, the subexponentiality for the compound Poisson distribution (CP for short) and the ID 
are characterized with three equivalent relations $(\mathrm{i})$-$(\mathrm{iii})$ above. We will 
review past researches in terms of this equivalence, taking $\cals,\,\cals_\Delta,\,\cals_{loc},\,\cals_{ac}$ in order.
In Table \ref{table1}, we summarize these past researches, where hatching cells implies that 
the subexponentiality was completely characterized in these cases and SD is the abbreviation for the 
self-decomposable distribution. 

$\cals$: the characterization for the distribution has been completed by 
\cite{Embrechts:Goldie:Veraverbeke:1979} and \cite{Pakes:2004}. 
%extended 
%this characterization to the two-sided cases. Thus, the characterization for the distribution has been completed. 

$\cals_{\Delta}$: Asmussen et al. \cite{Asmussen:Foss:Korshunov:2003} 
characterized two equivalences with some conditions 
for both the CP and the ID in one-sided case. This was extended 
to the three equivalences  
by Wang et al. \cite{wang:cheng:wang:2005} with different conditions.  
In the general two-sided case Watanabe and Yamamuro 
\cite{Watanabe:Yamamuro:2009} showed three equivalences  
with less conditions than before. 

$\cals_{loc}$: completely characterized in the one-sided ID case 
(Watanabe and Yamamuro \cite{Watanabe:Yamamuro:2010}).

$\cals_{ac}$: in the positive-half case, 
Shimura and Watanabe \cite{shimura:watanabe:2022} treated the CP, while Watanabe \cite{Watanabe:2020} studied the subexponential density on the ID. 
The difference between the CP and the ID is that the treatment of problems arising from the behavior of 
L\'evy measures around the origin which explodes in the ID case. 
Recently the two-sided cases in both the CP and the ID were investigated by \cite{matsui:2022}, where a kind of 
monotonic condition is shown to be necessary. 

In addition for the class of self-decomposable distributions it is known that the three equivalent relations hold 
without any conditions in all definitions of subexponentiality (see Theorem 1.3 and Lemma 4.2 in \cite{Watanabe:Yamamuro:2010}).

After the review of past researches, now we revisit our achievement in detail. 
Recall that our main focus in on $\cals_\Delta$ and $\cals_{loc}$. We 
complete the characterization of both $\cals_\Delta$ and $\cals_{loc}$ in the positive-half CP case without additional assumptions.  
In the two-sided case, we characterize three equivalences for $\cals_\Delta$ 
with less conditions than before, particularly we remove the long tailed assumption on 
L\'evy measure, which was assumed in the preceding results \cite{Asmussen:Foss:Korshunov:2003,Watanabe:Yamamuro:2009}. 
As a by-product, we recovered the classical result of Embrechts and Hawkes \cite{Embrechts.Hawkes:1982} 
in the one-sided case and extended it to that for the two-sided case. 
We also study $\cals_{loc}$ in the two-side case, which has not been investigated at all.

For IDs we characterize $\cals_\Delta$ with different conditions, i.e. 
we replace the moment condition of \cite{Watanabe:Yamamuro:2009} with the al.d. condition on increments of 
both the distribution and the L\'evy measure, with which we only need to check the positive-half sides.  
Moreover, based on the two monotonic-type conditions, we derive new characterization of 
 $\cals_\Delta$ without using the long-tailed assumption on the L\'evy measure. 
With this result we newly characterize $\Delta$-subexponentiality of the $s$-self-decomposable distribution. 
Again our study of $\cals_{loc}$ in the two-sided case has not been investigated before.  
In summary, we achieve a major progress in the characterization problems of subexponentialities 
on IDs.

We close this section with definitions need in the latter sections. 
We are starting with the monotonic-type conditions tailored for each definition of subexponentiality.  
%Firstly we give the ``asymptotic to a non-increasing function'' (a.n.i.) condition. }
%Definitions related with ``asymptotic to a non-increasing function'' (a.n.i. for short) and 
%``almost decreasing'' (al.d. for short). 
\begin{definition}
We say that a function $\alpha:\R\to \R_+$ is asymptotic to a non-increasing function $($a.n.i.$)$ denoted by 
$\alpha \in \cala$
%belongs to the class $\cala$ of asymptotic to a non-increasing function 
if  
$\alpha$ is locally bounded and positive on $[x_0,\infty)$ for some $x_0>0$, and 
\begin{align}
\label{eq:def:ani}
 \sup_{t\ge x} \alpha(t)\sim \alpha(x)\quad \text{and}\quad \inf_{x_0\le t \le x} \alpha(t)\sim \alpha(x). 
\end{align}
$(\mathrm{i})$ $F$ belongs to the class $\cala_{\Delta}$ if $F(x+\Delta) \in \cala$. \\
$(\mathrm{ii})$ $F$ belongs to the class $\cala_{loc}$ if $F(x+\Delta) \in \cala$ for all $\Delta$. \\
$(\mathrm{iii})$ $F$ belongs to the class $\cala_{ac}$ if $F$ is absolutely continuous and its density $f\in \cala$. 
\end{definition}
%The second one is the ``almost decreasing'' (al.d.) condition. 
\begin{definition}
 We say that a function $\alpha:\R\to \R_+$ is almost decreasing $($al.d.$)$ denoted by 
$\alpha \in \cald$ 
%belongs to the class $\cald$ of almost decreasing 
if there exists $x_0>0$ and $K>0$ such that 
\begin{align}
\label{eq:def:ald}
 \alpha(x+y) \le K\alpha(x)\quad \text{for all}\ x>x_0,\,y>0.
\end{align} 
$(\mathrm{i})$ $F$ belongs to the class $\cald_{\Delta}$ if $F(x+\Delta) \in \cald$. \\
$(\mathrm{ii})$ $F$ belongs to the class $\cald_{loc}$ if $F(x+\Delta) \in \cald$ for all $\Delta$. \\
$(\mathrm{iii})$ $F$ belongs to the class $\cald_{ac}$ if $F$ is absolutely continuous and its density $f\in \cald$. 
\end{definition}
Obviously $\cala_{ac} \subset \cala_{loc}\subset \cala_\Delta$ and $\cald_{ac} \subset \cald_{loc} \subset \cald_\Delta$ and 
we have following lemma, whose proof is elementary and omitted.  
\begin{lemma}
\label{lem:ald:symp:equiv}
 If $\alpha$ is al.d. and $\alpha(x)\sim \beta(x)$, then $\beta$ is al.d.
\end{lemma}
%We will investigate properties of the above sort, particularly on infinitely divisible distributions $\mu$ on $\R$. 
%The characteristic function (ch.f.) of $\mu$ is 
The precise definition of the ID $\mu$ is given with its characteristic function (ch.f.) denoted by $\wh \mu$,  
\begin{align}
\label{def:chf:idr}
 \wh \mu (z) = \exp \Big\{
\int_{-\infty}^\infty 
(e^{izy}-1-i zy{\bf 1}_{\{ |y|\le 1 \}} ) \nu (dy) +iaz -\frac{1}{2} b^2 z^2
\Big\},
\end{align}
where $a \in \R,\,b\ge 0$ and $\nu$ is the L\'evy measure satisfying $\nu(\{0\})=0$ and $\int_{-\infty}^\infty (1 \wedge x^2) \nu (dx) <\infty$.
We denote by $\id$ and $\idrp$ the class of all IDs on $\R$ and that of all IDs on $\R_+$, respectively. 

Other miscellaneous definitions are as follows. 
Let $f_+$ be the density of the conditional distribution $F_+$ of $F$ on $\R_+$:
$f_+(x)={\bf 1}_{\R_+}(x) f(x) / \ov F(0),\,x \in \R$. $F^+$ is the distribution given by 
$F^+(x)=F(x)$ for $x\ge 0$ and $F^+(x)=0$ for $x<0$.  
For non-negative functions $\alpha(x)$ and $\beta(x)$, a notation $\alpha(x)\asymp \beta(x)$ implies 
that 
\[
 0<\liminf_{x\to\infty} \big(\alpha(x)/ \beta(x)\big) \le \limsup_{x\to\infty} \big( \alpha(x)/\beta(x) \big) <\infty. 
\]
The set of integers and the set of non-negative integers are respectively 
 denoted by $\Z$ and $\Z_+$.

\section{Relations between Subexponential definitions under $\cala$ and $\cald$}
\label{sec:four:relations}
In the following two lemmas we provide equivalent conditions of  
%characterize the relations of 
%various versions of 
a.n.i. assumptions, %$\cala$ 
and characterize a relation between $\cala_{ac}$ and $\call_{ac}$, respectively.  
Then we investigate the relations between subexponential definitions $\cals_\Delta,\cals_{loc}$ and $\cals_{ac}$. 
\begin{lemma}%[c.f. Lemma 3.8 of \cite{matsui:2022}]
\label{lem:equiv:ani:function}
 Let $\alpha (x) \in \cala$, and then  
\begin{align}
\label{lem:ani:function}
\text{there exists a positive non-increasing function}\ \beta\ \text{such that}\ \alpha(x)\sim \beta(x). 
%\lim_{x\to\infty} \alpha(x)/\beta (x)=1. 
\end{align}
Conversely, if $\alpha (x)\to 0$ as $x\to \infty$, then \eqref{lem:ani:function} implies $\alpha \in \cala$. 
In particular, $f\in\call$ with the condition \eqref{lem:ani:function} implies that $f \in \cala$, and 
$F\in \cala_{\Delta}$ if and only if $F(x+\Delta)$ satisfies \eqref{lem:ani:function}, so that 
$F\in \cala_{loc}$ if and only if $F(x+\Delta)$ satisfies \eqref{lem:ani:function} for all $\Delta$. 
\end{lemma}
\begin{lemma}
\label{lem:ldel+ani}
 $F \in \call_\Delta \cap \cala_{ac}$ implies $F \in \call_{ac}$. 
\end{lemma}
Generally, neither $F_+ \in \cals_{\Delta}$ nor $F^+ \in \cals_{\Delta}$ may imply $F \in \cals_{\Delta}$, which is the same as that 
$F_+ \in \cals_{ac}$ or $F^+ \in \cals_{ac}$ may not imply $F \in \cals_{ac}$. Under the al.d. assumption  
$F^+\in \cals_{\Delta}$ is known to imply $F\in \cals_\Delta$ \cite[Lemma 27]{Foss:Korshunov:Zachary:2013}. 
The same can hold with $F_+\in \cals_{\Delta}$. 
%in Definition
%\ref{eq:def:ald}, we can show these reverse directions. 
\begin{proposition}
\label{prop:sdel+ani}
 Let $F\in \cald_\Delta$, then $F^+ \in \cals_\Delta \Leftrightarrow F_+\in \cals_\Delta \Leftrightarrow F\in \cals_\Delta$, so that 
under $F\in \cald_{loc}$, $F^+ \in \cals_{loc} \Leftrightarrow F_+\in \cals_{loc} \Leftrightarrow F\in \cals_{loc}$. 
\end{proposition}
Since $\cala_{\Delta} \subset \cald_\Delta$ and $\cala_{loc} \subset \cald_{loc}$, the results in Proposition \ref{prop:sdel+ani} 
follow under $\cala_{\Delta}$ and $\cala_{loc}$, respectively.

Now we characterize the relations between subexponential definitions.
\begin{proposition}
 \label{prop:sac+ani}
Assume that $F$ has a density. 
Under $F \in \cala_{ac}$, the following are equivalent. 
\[
 F_+ \in \cals_\Delta,\,F^+ \in \cals_\Delta,\, F\in \cals_\Delta,\, F_+ \in \cals_{loc},\, F^+ \in \cals_{loc},\, F\in \cals_{loc}, \,
F_+ \in \cals_{ac},\, F^+ \in \cals_{ac},\, F\in \cals_{ac}. 
\]
If $F$ is a distribution on $\R_+$, then under $F\in \call_{ac}$ the following are equivalent.  
\[
 F\in \cals_\Delta,\,F\in \cals_{loc},\,F\in \cals_{ac}. 
\]
\end{proposition}

\section{Properties of $\cals_\Delta$ and $\cals_{loc}$}
\label{section:prop:sdelta:sloc}
In this section we focus on what we call ``usual properties'' of 
$\call_\Delta$ and $\cals_\Delta$ which are common in the literature 
of the subexponentiality and which are useful to derive the main theorems. 
One of the central problems there is 
the closedness property of $\call_\Delta$ and $\cals_{\Delta}$
under various operations. 
In the positive-half case 
some of these properties were investigated by \cite[Sections 4.5-4.8]{Foss:Korshunov:Zachary:2013}. 
Since $\call_\Delta$ is unchanged in both one- and two-sided cases, we concentrate on properties for $\cals_\Delta$. 
We derive several new results (versions of convolution root and factorization) and moreover, 
extend all results of \cite[Subsection 4.6]{Foss:Korshunov:Zachary:2013} into the two-sided distribution. 
In the derivation, both a.n.i. and al.d. assumptions are exploited. 

We start with an extension of \cite[Theorem 4.21]{Foss:Korshunov:Zachary:2013}. 
%\textcolor{red}{
%\begin{lemma}
%\label{lem:ifandonlyif:sub:delta+}
% Suppose that the distribution $F$ on $\R_+$ is such that $F\in \call_\Delta$ for some $\Delta$. 
%Let $\xi_1$ and $\xi_2$ be two independent r.v.'s with common distribution $F$. 
%Then the following assertions are equivalent: \\
%$(\mathrm{i})$ $F \in \cals_\Delta$. \\
%$(\mathrm{ii})$ For every function $\alpha$ s.t. $\alpha(x)\to \infty$ as $x\to\infty$ and $\alpha(x)<x/2$.
%\begin{align}
%\label{ifandonlyif:sub:delta+}
%\P\big(\xi_1+\xi_2 \in x+\Delta,\,\xi_1 > \alpha(x),\,\xi_2>\alpha(x) \big) &= o(F(x+\Delta)).  
%\end{align} 
%$(\mathrm{iii})$ There exists a function $\alpha$ such that $\alpha(x)<x/2,\,\alpha(x)\to\infty$ as $x\to\infty$, 
%the function $F(x+\Delta)$ is $\alpha$-insensitive and 
%\begin{align}
%\label{ifandlonly:sub:delta+:int}
% \int_{(\alpha(x),x-\alpha(x)]}F(x+\Delta-y)F(dy)=o(F(x+\Delta))
%\end{align}
%holds.
%and \eqref{ifandonlyif:sub:delta+} holds.  
%\end{lemma}}
\begin{lemma}
\label{lem:ifandonlyif:sub:delta}
 Suppose that $F \in \call_\Delta$ for some $\Delta$. 
Let $\xi_1$ and $\xi_2$ be two independent r.v.'s with common distribution $F$. 
Then the following assertions are equivalent: \\
$(\mathrm{i})$ $F \in \cals_\Delta$. \\
$(\mathrm{ii})$ For every function $\alpha$ such that $\alpha(x)\to \infty$ as $x\to\infty$ and $\alpha(x)<x/2$, 
\begin{align}
\label{ifandonlyif:sub:delta:real}
 \P\big(\xi_1+\xi_2 \in x+\Delta,\,\xi_1 \le -\alpha(x)\big) = \int_{(-\infty,-\alpha(x)]}F(x+\Delta-y)F(dy) &= o(F(x+\Delta)),  \\
\text{and}\qquad  \P\big(\xi_1+\xi_2 \in x+\Delta,\,\xi_1 > \alpha(x),\,\xi_2>\alpha(x) \big) &= o(F(x+\Delta)).  
\label{ifandonlyif:sub:delta:common}
\end{align} 
$(\mathrm{iii})$ There exists a function $\alpha$ such that $\alpha(x)<x/2,\,\alpha(x)\to\infty$ as $x\to\infty$, 
$F(x+\Delta)$ is $\alpha$-insensitive and both \eqref{ifandonlyif:sub:delta:real} and \eqref{ifandonlyif:sub:delta:common} hold. 

If additionally $F$ is a distribution on $\R_+$, then equivalence of the three holds without the condition \eqref{ifandonlyif:sub:delta:real} 
in $(\mathrm{ii})$ and $(\mathrm{iii})$. Moreover, we could replace the condition \eqref{ifandonlyif:sub:delta:common} of $(\mathrm{iii})$ with 
  \begin{align}
\label{ifandlonly:sub:delta+:int}
 \int_{(\alpha(x),x-\alpha(x)]}F(x+\Delta-y)F(dy)=o(F(x+\Delta)). 
\end{align}
\end{lemma}

The second one is the closedness of $\cals_\Delta$ under factorization. 
\begin{proposition}[factrization]
\label{prop:factrization:delta}
 Let $H=F\ast G \in \cals_{\Delta}$ and $G(x+\Delta)=o(H(x+\Delta))$. 
If $F\in \cala_{\Delta}$ and $[$$H\in \cald_{\Delta}$ or $G(x+\Delta)=o(F(x+\Delta))$$]$, then $F\in \cals_{\Delta}$.
\end{proposition}

Next we show several versions of 
[the closedness under convolution root] with a.n.i. assumptions. 

\begin{proposition}[convolution root: positive-half]
\label{prop:convroot:sdel+ani}
Let $F$ be a distribution on $\bbr_+$.
If $F\in \cala_\Delta$ and $F^{*N} \in  \cals_{\Delta}$, then $F \in  \cals_{\Delta}$. Thus, 
if $F\in \cala_{loc}$ and $F^{*N} \in  \cals_{loc}$, then $F \in  \cals_{loc}$.
\end{proposition}
%\begin{proposition}
%\label{prop:convroot:sden+ani}
%Let $f$ be a density function on $\bbr_+$.
 %Let $f\in \cala$. If $f^{*N} \in  \cals$, then $f \in  \cals$.
%\end{proposition}
\begin{proposition}[convolution root]
\label{prop:convroot:sdel+ani:two-side}
If $F\in \cala_\Delta$ and $F^{*N} \in  \cals_{\Delta}\cap \cald_{\Delta}$, then $F \in  \cals_{\Delta}$. 
Consequently, if $F\in \cala_{loc}$ and $F^{*N} \in  \cals_{loc}\cap \cald_{loc}$, then $F \in  \cals_{loc}$.
\end{proposition}
\begin{corollary}[convolution root for density]
\label{prop:convroot:sden+ani}
 Let $f\in \cala$. If $f^{*N} \in  \cals\cap\cald$, then $f \in  \cals$.
\end{corollary}
\begin{remark}
\label{rem:conv:root}
 $(\mathrm{i})$ 
There are many studies about closedness of subexponentiality with the convolution root operation. 
The closedness of $\cals$ for positive-half distributions was shown by \cite{Embrechts:Goldie:Veraverbeke:1979} and 
that for two-sided ones was proved by \cite{Watanabe:2008}. However, the classes $\cals_{ac}, \cals_{\Delta}$ 
and $\cals_{loc}$ are known not to be closed under convolution roots, see \cite[Corollary 1.1]{Watanabe:Yamamuro:2017} 
and of \cite[Corollary 1.1]{Watanabe:2019}. 
Then the next natural question has been that under what conditions the closedness holds. 
Propositions \ref{prop:convroot:sdel+ani}, \ref{prop:convroot:sdel+ani:two-side} and Corollary \ref{prop:convroot:sden+ani}
give some solutions to the problem. 
Notice that the closedness of under convolution roots often holds within the class of compound Poisson distributions 
or IDs. We discuss this topic with unsolved future problems in Section \ref{sec:id}, 
see Remark \ref{rem:conv:root:cp:id}. 
\\
%Other classes such as $\cals_\gamma,\,\gamma>0$ and 
%$\mathcal{OS}$ are known to be not closed generally, but closed within the class of infinitely divisible distributions.  \\
 $(\mathrm{ii})$ Corollary \ref{prop:convroot:sden+ani} is an improvement of \cite[Theorem 3.13]{matsui:2022}, where 
either $f^{\ast k} \in \cala$ for all $1\le k \le N$, or $f\in \call$ is required. 
\end{remark}
The following results are two-sided extension of one-sided properties: Theorems 4.22, 4.23, Corollary 4.24 and Theorem 4.25 of 
\cite{Foss:Korshunov:Zachary:2013}. In particular Kesten bound (Proposition \ref{prop:kesten}) is useful to derive results in the following sections. 
\begin{proposition}[asymptotic equivalence]
\label{prop:asympt:equiv}
Suppose that $F \in \cals_{\Delta} \cap \cald_\Delta$, %and that $F(x+\Delta)$ is al.d. Suppose that 
$G\in \call_\Delta$ and $G(x+\Delta)\asymp F(x+\Delta) $ as $x \to \infty$. 
Then  $G \in \cals_{\Delta}$. In particular, if $G(x+\Delta)\sim cF(x+\Delta) $ as $x \to \infty$ for some $c >0$, then $G \in \cals_{\Delta}$.  
Moreover, if additionally $G$ is a distribution on $\R_+$, then the  assertion holds without assumption $\cald_\Delta$ on $F$. 
\end{proposition}
%For a combination [$F$ on $\R$ and $G$ on $\R_+$], the assumption $\cald_\Delta$ is removed in 
%the proposition. Since we use this fact in e.g. the proof of Proposition \ref{prop:convroot:sdel+ani:two-side}, 
%we state this as a lemma. The proof is done by noticing that $F \in \cals_\Delta$ implies $F_+\in \cals_\Delta$ 
%and applying \cite[Theorem 4.22]{Foss:Korshunov:Zachary:2013}. 
%\begin{lemma}
% Suppose that $F \in \cals_\Delta$ and $G$ is a distribution on $\R_+$ such that 
%$G\in \call_\Delta$ and $G(x+\Delta)\asymp F(x+\Delta) $ as $x \to \infty$. Then $G \in \cals_{\Delta}$. 
%In particular,  if $G(x+\Delta)\sim cF(x+\Delta) $ as $x \to \infty$ for some $c >0$, then $G \in \cals_{\Delta}$. 
%\end{lemma}
\begin{proposition}[convolution]
\label{prop:asympt:equiv2}
Suppose that $F \in \cals_{\Delta}\cap \cald_{\Delta}$. % and that $F(x+\Delta)$ is al.d.. 
Let $G_1,G_2$ be two distributions such that $G_1(x+\Delta)/F(x+\Delta) \to c_1$ and $G_2(x+\Delta)/F(x+\Delta) \to c_2$ as $x \to \infty$
for some constants $ c_1, c_2 \geq 0$. Then
\begin{equation}
\label{aymp:eqiv:doubleG}
 \frac{G_1*G_2 (x+\Delta)}{F(x+\Delta)} \to c_1+ c_2\quad \text{as}\quad x \to \infty.
\end{equation}
Further, if $c_1+ c_2>0$
then $G_1*G_2 \in \cals_{\Delta}$.
\end{proposition}
Following corollary is immediately derived by Proposition \ref{prop:asympt:equiv2}. 
\begin{proposition}
\label{prop:asympt:equiv3}
Suppose that $F \in \cals_{\Delta}\cap \cald_\Delta$. % and that $F(x+\Delta)$ is al.d.. 
Let $G$ be a distribution such that $G(x+\Delta)/F(x+\Delta) \to c \ge 0$ as $x \to \infty$. 
Then, for any $n \ge 2$, $G^{*n}(x+\Delta)/F(x+\Delta) \to nc \ge 0$ as $x \to \infty$. If $c >0$, then $G^{*n} \in \cals_{\Delta}$.
\end{proposition}
\begin{proposition}[Kesten bound]
\label{prop:kesten}
Suppose that $F \in \cals_{\Delta} \cap \cald_\Delta$. Then, for any $\vep >0$, there exists  $x_0=x_0(\vep) >0 $ 
and $c(\vep)>0 $ such that, for any $x > x_0$ and for any $n \ge 1$,
$$F^{*n}(x+\Delta)\le c(\vep)(1+\vep)^nF(x+\Delta).$$
\end{proposition} 

The last two results are about the closedness under factorization 
when a negligible distribution has a point mass at the origin.
The negligible distribution includes a compound Poisson distribution.
\begin{proposition}[factrization]
\label{prop:decomp:delta}
Let $H=G*F\in \cals_{\Delta}$ such that $G(dx)= p_G\delta_0(dx) +(1-p_G)G_1(dx)$ with 
$p_G \in (2^{-1},1)$ where $G_1$ is a distribution. Suppose that $H\in \cald_\Delta$ or $[\,F,G,H$ are distributions on $\R_+\,]$. 
Then, $G(x+\Delta)=o(H(x+\Delta))$ implies $ F\in \cals_{\Delta}$, 
and indeed $H(x+\Delta) \sim F(x+\Delta).$
\end{proposition}
\begin{proposition}[factrization: compound Poisson]
\label{prop:decomp:delta:Poisson}
Let $H=G*F\in \cals_{\Delta}$ such that $G$ is a compound Poisson distribution. 
Suppose that $H\in \cald_{\Delta}$ or $[\,F,G,H$ are distributions on $\R_+\,]$. 
Then, $G(x+\Delta)=o(H(x+\Delta))$ implies $ F\in \cals_{\Delta}$, and indeed $H(x+\Delta) \sim F(x+\Delta).$
\end{proposition}

\section{Main results}
\label{sec:main}
We investigate $\cals_\Delta$ and $\cals_{loc}$ separately in the 
compound Poisson case (Section \ref{compund:poisson}) and in 
the infinitely divisible case (Section \ref{sec:id}). 
\subsection{Compound Poisson case}
\label{compund:poisson}
Let $G$ be a distribution on $\R$. For $\lambda>0$ define the compound Poisson probability measure by 
\begin{align}
\label{def:cp:measure}
  \mu(dx)= e^{-\lambda}\sum_{n=0}^\infty (\lambda^n/n!)G^{\ast n}(dx), 
\end{align}
where $G^{\ast 0}$ is the Dirac measure. 
%\begin{align*}
%\label{def:cp:distribution}
% \mu(dx)= e^{-\lambda} \delta_0(dx)+(1-e^{-\lambda})\mu_0(dx),
%\end{align*}
%where 
%\begin{align}
%\label{eq:cp:proper:distribution:part}
% \mu_0(dx)= (e^\lambda-1)^{-1} \sum_{n=1}^\infty (\lambda^n/n!)G^{\ast n}(dx). 
%\end{align}
\begin{theorem}
\label{thm:cp:twosided}
Let $\mu$ be a compound Poisson given by \eqref{def:cp:measure}. 
%with L\'evy measure $\lambda G$ where $\lambda>0$ and $G$ is a distribution. 
 The following are equivalent.\\
$(\mathrm{i})$ $\mu \in \cals_\Delta \cap \cald_\Delta$, \quad $(\mathrm{ii})$ $G \in \cals_\Delta \cap \cald_\Delta$, \quad $(\mathrm{iii})$ $G \in \call_\Delta \cap \cald_\Delta$ and $\mu(x+\Delta)/G(x+\Delta)\sim \lambda$.\\
%$\lim_{x\to\infty} \mu(x+\Delta)/G(x+\Delta)=\lambda$. \\
If additionally $G$ is a distribution on $\R_+$, then the following assertions are equivalent. \\
$(\mathrm{i})$ $\mu \in \cals_\Delta$, \quad $(\mathrm{ii})$ $G \in \cals_\Delta$,
\quad $(\mathrm{iii})$ $G \in \call_\Delta$ and $\mu(x+\Delta)/G(x+\Delta) \sim \lambda$.
%$\lim_{x\to\infty} \mu(x+\Delta)/G(x+\Delta)=\lambda$. 
\end{theorem}
An immediate corollary follows from the fact $\cala_\Delta \subset \cald_\Delta$. 
\begin{corollary}
\label{col:cp:twosided:ani}
 Let $\mu$ be a compound Poisson given by \eqref{def:cp:measure}. The following are equivalent.\\
$(\mathrm{i})$ $\mu \in \cals_\Delta \cap \cala_\Delta$, \quad $(\mathrm{ii})$ $G \in \cals_\Delta \cap \cala_\Delta$, 
\quad$(\mathrm{iii})$ $G \in \call_\Delta \cap \cala_\Delta$ and $\mu(x+\Delta)/G(x+\Delta)\sim \lambda$.
%$\lim_{x\to\infty} \mu(x+\Delta)/G(x+\Delta)=\lambda$. 
\end{corollary}
In the following corollary  
we recovered the classical one-sided result of Embrechts and Hawkes \cite{Embrechts.Hawkes:1982},
 and moreover extended it to the characterization of the two-sided distribution.
\begin{corollary}
\label{cor:ineger:twoside}
Let $\Delta=(0,1]$ and let  $\mu$ be an ID on $\mathbb Z$ with L\'evy measure 
$\lambda G$ where $\lambda>0$ and $G$ is a distribution. 
 The following assertions are equivalent.\\
$(\mathrm{i})$ $\mu \in \cals_\Delta \cap \cald_\Delta$,
\quad $(\mathrm{ii})$ $G \in \cals_\Delta \cap \cald_\Delta$, 
\quad$(\mathrm{iii})$ $G \in \call_\Delta \cap \cald_\Delta$ and $\mu(x+\Delta)/G(x+\Delta)\sim \lambda$. \\
%$\lim_{x\to\infty} \mu(x+\Delta)/G(x+\Delta)=\lambda$. \\
If additionally $G$ is a distribution on $\Z_+$, then the following assertions are equivalent. \\
$(\mathrm{i})$ $\mu \in \cals_\Delta$, 
\quad $(\mathrm{ii})$ $G \in \cals_\Delta$, 
\quad $(\mathrm{iii})$ $G \in \call_\Delta$ and $\mu(x+\Delta)/G(x+\Delta)\sim \lambda$.
%$\lim_{x\to\infty} \mu(x+\Delta)/G(x+\Delta)=\lambda$. 
\end{corollary}
%\begin{remark}
%In Corollary \ref{cor:ineger:twoside} we recovered  
%The one-sided case of the above corollary is the same as 
%the classical result of Embrechts and Hawkes \cite{Embrechts.Hawkes:1982} 
%in one-sided case and also extended it to that for the two-sided case.
%\end{remark}
The results in Theorem \ref{thm:cp:twosided} are immediately extended to 
those for the local subexponentiality. Notice that the two-sided case below has not been studied before. 
\begin{corollary}
\label{coro:cp:local}
Let $\mu$ be a compound Poisson with L\'evy measure $\lambda G$ where $\lambda>0$ and $G$ is a distribution. 
 The following assertions are equivalent.\\
$(\mathrm{i})$ $\mu \in \cals_{loc} \cap \cald_{loc}$, \quad $(\mathrm{ii})$ $G \in \cals_{loc}  \cap \cald_{loc} $,  \\
$(\mathrm{iii})$ $G \in \call_{loc}  \cap \cald_{loc} $ and $\mu(x+\Delta)/G(x+\Delta)\sim \lambda$
%$\lim_{x\to\infty} \mu(x+\Delta)/G(x+\Delta)=\lambda$
for any $\Delta$. \\
If additionally $G$ is a distribution on $\R_+$, then the following assertions are equivalent. \\
$(\mathrm{i})$ $\mu \in \cals_{loc}$, 
\quad $(\mathrm{ii})$ $G \in \cals_{loc}$,
\quad $(\mathrm{iii})$ $G \in \call_{loc}$ and $\mu(x+\Delta)/G(x+\Delta)\sim \lambda$ for nay $\Delta$.
%$\lim_{x\to\infty} \mu(x+\Delta)/G(x+\Delta)=\lambda$ for nay $\Delta$. 
\end{corollary}

\begin{remark}
The positive distribution part of Corollary \ref{coro:cp:local} has two different proofs. 
One is an immediate consequence of the positive part of Theorem \ref{thm:cp:twosided}, and 
the other pass is given by a more general result \cite[Theorem 1.1]{Watanabe:Yamamuro:2010}. 
We briefly explain this. For $(\mathrm{ii}) \Leftrightarrow (\mathrm{iii})$ observe that by Proposition \ref{prop:asympt:equiv},  
$G\in \cals_{loc} \Leftrightarrow G_1 \in \cals_{loc}$ where $G_{1}$,  
given in the proof of Theorem \ref{thm:cp:twosided}, 
corresponds to the conditional L\'evy 
measure $\nu_{(1)}$ of \cite[Theorem 1.1]{Watanabe:Yamamuro:2010}. Thus $(2) \Leftrightarrow (3)$ in \cite[Theorem 1.1]{Watanabe:Yamamuro:2010} yields 
$(\mathrm{ii}) \Leftrightarrow (\mathrm{iii})$ where the second condition of $(3)$ is $\mu(x+\Delta)\sim \lambda G(x+\Delta)$ for all $\Delta$. 
The part $(\mathrm{i}) \Leftrightarrow (\mathrm{ii})$ directly follows from $(1) \Leftrightarrow (2)$ of \cite[Theorem 1.1]{Watanabe:Yamamuro:2010}
since $G\in \cals_{loc} \Leftrightarrow G_1 \in \cals_{loc}$. 
\end{remark}

\subsection{Infinitely divisible case}
\label{sec:id}
Define the normalized L\'evy measure $\nu_{(1)}$ as $\nu_{(1)}(dx)={\bf 1}_{(1,\infty)}(x)\nu(dx)/\nu((1,\infty))$.
\begin{theorem}
\label{thm:ID:delta:subexponential}
Let $\mu \in \id$. %be an infinitely divisible distribution satisfying  (1.3). 
Under the condition $\nu_{(1)} \in  \call_\Delta $, the following  
are equivalent.\\
$(\mathrm{i})$ $\mu \in \cals_\Delta \cap \cald_\Delta$, 
\quad $(\mathrm{ii})$ $\nu_{(1)} \in \cals_\Delta \cap \cald_\Delta$,  
\quad$(\mathrm{iii})$ $\nu_{(1)} \in \call_\Delta \cap \cald_\Delta$ and $ \mu(x+\Delta)\sim \nu(x+\Delta)$. 
\end{theorem}
\begin{remark}
In Theorem \ref{thm:ID:delta:subexponential}, the moment condition $[\,\int_{\R} e^{-\vep y}\nu (dy)<\infty$ for some $\vep>0\,]$ in 
\cite[Theorem 1.2]{Watanabe:Yamamuro:2009} is removed, but one has to pay the 
price that the condition $\cald_\Delta$ is required in both $\mu$ and $\nu_{(1)}$. The merit is that we only 
need to check the positive-half side 
of $\mu$ or $\nu$. 
\end{remark}
Notice that the assumption $\nu_{(1)} \in  \call_\Delta $ of Theorem \ref{thm:ID:delta:subexponential} is too strong, 
and thus we derive the three equivalent relations in different conditions. 
\begin{theorem}
\label{thm:ID:delta:subexponential:cp}
 Let $\mu \in \id$ and for a constant $c_1>1$ let $\mu_{c_1}$ be a compound Poisson with 
the L\'evy measure $\nu_{c_1}(dx)={\bf 1}_{\{c_1 <x\}}\nu (dx)$
Suppose that there exists $c_1>1$ such that $\mu_{c_1}\in \cala_\Delta$. 
%Suppose that there exists $c_1>1$ such that 
%$\mu_{c_1}\in \idrp$ with $a=b=0$ and $\nu_{c_1}(dx)={\bf 1}_{\{c_1 <x\}}\nu (dx)/\nu((c_1,\infty))$ satisfies 
%$\mu_{c_1}\in \cala_\Delta$. 
 Then following assertions 
are equivalent.\\
$(\mathrm{i})$ $\mu \in \cals_\Delta \cap \cald_\Delta$, 
\quad $(\mathrm{ii})$ $\nu_{(1)} \in \cals_\Delta $,  
\quad$(\mathrm{iii})$ $\nu_{(1)} \in \call_\Delta \cap \cald_\Delta $ and $ \mu(x+\Delta)\sim \nu(x+\Delta)$. 
\end{theorem}
Now we could treat the $s$-self-decomposable distribution which is a more general class in $\id$ than 
the self-decomposable distribution (\cite{Jurek:1985}), namely its L\'evy density $g$ of $\nu(dx)=g(x)dx$ 
satisfies that $g(x)$ is non-increasing on $(0,\infty)$ and is non-decreasing on $(-\infty,0)$
 (\cite[Theorem 2.2 (b)]{Jurek:1985}). 
\begin{theorem}
\label{thm:ID:subexponential:sself}
 Let $\mu \in \id$ be an $s$-self-decomposable distribution with the L\'evy measure $\nu(dx)$.  %=g(x)dx$, then 
 Then following assertions 
are equivalent.\\
$(\mathrm{i})$ $\mu \in \cals_\Delta \cap \cald_\Delta$, 
\quad $(\mathrm{ii})$ $\nu_{(1)} \in \cals_\Delta $,  
\quad$(\mathrm{iii})$ $\nu_{(1)} \in \call_\Delta $ and $ \mu(x+\Delta)\sim \nu(x+\Delta)$. 
\end{theorem}
We state the result for $\cals_{loc}$ 
corresponding to Theorem \ref{thm:ID:delta:subexponential}. 
\begin{theorem}
\label{thm:ID:loc:subexponential}
Let $\mu \in \id$. The following assertions are equivalent.\\
$(\mathrm{i})$ $\mu \in \cals_{loc} 
 \cap \cald_{loc}$, 
\quad $(\mathrm{ii})$ $\nu_{(1)} \in \cals_{loc} 
 \cap \cald_{loc}$, \\
$(\mathrm{iii})$ $\nu_{(1)} \in \call_{loc} \cap \cald_{loc}$ and $ \mu(x+\Delta)\sim \nu(x+\Delta)$ for any $\Delta$.
\end{theorem}
\begin{remark} 
\label{rem:conv:root:cp:id}
Recall that the classes $\cals_{ac}, \cals_{\Delta} $ and $\cals_{loc} $ are not closed under convolution roots, see 
Remark \ref{rem:conv:root}.   
However, we find from Theorem \ref{thm:cp:twosided} and Corollary \ref{coro:cp:local} that 
the classes $\cals_{\Delta}$ in the one-sided case and $\cals_\Delta \cap \cald_\Delta$ are closed within 
compound Poisson distributions. Moreover, due to Theorem \ref{thm:ID:loc:subexponential} and \cite[Theorem 1.1]{Watanabe:Yamamuro:2010} 
, the classes $\cals_{loc}$ in one-sided case and $\cals_{loc} \cap \cald_{loc}$ are closed within 
the classes of $\idrp$ and $\id$ respectively. %infinitely divisible distributions.   
%and that  the class  $\cals_{loc}  \cap \cald_{loc}$ is closed within the class of the infinitely divisible distributions.  
Thus we raise a natural open question:\end{remark}  
\vspace{-2mm}
%\begin{itembox}[l]{Problem}
{\bf Problem}: Are the classes $\cals_{ac}$ and  $\cals_{\Delta}$  in the one-sided case and the classes $\cals_{ac} \cap \cald_{ac}$ and $\cals_\Delta \cap \cald_\Delta$ closed under convolution roots within the classes of $\idrp$ and $\id$ respectively $?$
%the infinitely divisible distributions ?
%\end{itembox}

\section{Application} 
\label{sec:application}

Needless to say, a random walk and a randomly stopped sum 
have applications in a variety of fields. We will observe that 
$\Delta$-subexponentiality is the characterized in the two classical themes. 
Thus, the same results hold for local subexponentiality.

\subsection{Supremum of a random walk}
We are starting with $\Delta$-subexponentiality of  
compound negative binomial distributions, which include compound geometric distributions  
as a subclass.  
\begin{theorem}
\label{theoremcompound}
Let $0<\lambda<1$,  $a>0$ and $G$ be a distribution. Define a compound negative binomial distribution $F_a$ as
\begin{equation}
 F_a:=\sum_{n=0}^{\infty}{a+n-1\choose a-1}(1-\lambda)^{a}\lambda^nG^{*n}. \nonumber
 \end{equation}
Then the following assertions are equivalent.\\
 $(\mathrm{i})$ $F_a \in \cals_\Delta \cap \cald_\Delta$, \quad $(\mathrm{ii})$ $G \in \cals_\Delta \cap \cald_\Delta$, \\
$(\mathrm{iii})$ $G \in \call_\Delta \cap \cald_\Delta$ and $ F_a (x+\Delta)\sim a\lambda/(1-\lambda) G (x+\Delta)$. \\
If additionally $G$ is a distribution on $\R_+$, then the following assertions are equivalent, \\
$(\mathrm{i})$ $F_a \in \cals_\Delta$, 
\quad $(\mathrm{ii})$ $G \in \cals_\Delta$, 
\quad $(\mathrm{iii})$ $G \in \call_\Delta$ and $ F_a(x+\Delta)\sim a\lambda/(1-\lambda) G(x+\Delta)$.
\end{theorem}

Next, we define the distribution $\pi$ of the supremum of a random walk. Let $\{X_n\}_{n=1}^{\infty}$ be iid random variables with the common distribution $G$ on $\mathbb R$. Let $\{S_n\}_{n=0}^{\infty}$ be a random walk on $\mathbb R$ defined by $S_0:=0$ and $S_n:=\sum_{k=1}^{n}X_k$ for $n \geq 1$. Let $\pi$ be the distribution of the supremum $M$ of  $\{S_n\}$, that is, $M:=\sup_{n \geq 0} S_n.$ 
Denote by ${\bf 1}_{(0,\infty)}(x)$ the indicator function of the set $(0,\infty).$ Define the measure $\nu$ on $(0,\infty)$ and a quantity $B$ as 
\begin{equation}
\nu(dx):={\bf 1}_{(0,\infty)}(x)\sum_{n=1}^{\infty}n^{-1}G^{*n}(dx) \nonumber
\end{equation}
and
$$ B:=\sum_{n=1}^{\infty}n^{-1}P(S_n> 0)=\nu((0,\infty)).$$
It is well known that $M < \infty$ a.s. if and only if $B < \infty$ and that if  $B < \infty$, then $\pi$ is a compound Poisson distribution on $\mathbb R_+$ with L\'evy measure $\nu$. A sufficient condition for $M < \infty$ a.s. is that $-\infty < \E [X_1] < 0$,  
(for more details, see \cite[Section XII.7]{Feller:2008}). 
Define $\lambda$ as $\lambda:=1-e^{-B}$ when $B< \infty$. 
We denote by $\lambda G_0$ the defective 
distribution of the first ascending ladder height in the random walk $\{S_n\}$. 
%Let \textcolor{red}{$Z^+$} be the first ascending ladder height in the random walk $\{S_n\}$ and denote the defective 
%distribution on $\mathbb R_+$ by $\lambda G_0$. 
Here, $G_0$ is a distribution on $\mathbb R_+$ with $ G_0(\{0\})=0$ given by  
\[
 G_0= -\lambda^{-1} \sum_{n=1}^\infty \big( (-1)^n/n! \big) \nu^{\ast n} \quad \Leftrightarrow \quad 
\nu= \sum_{n=1}^\infty (\lambda^n/n) G_0^{\ast n}. 
\]
It is also well known that 
\begin{equation}
\pi= \sum_{n=0}^{\infty}(1-\lambda)\lambda^n G_0^{*n}, \nonumber
\end{equation}
which is a compound geometric distribution, $(a=1)$ in $F_a$ (see, e.g. \cite[Section 5.3]{Foss:Korshunov:Zachary:2013} or 
\cite[Chapter X, 3.]{AsumussenAlbrecher2010}).
The distribution $\pi$ is important in classical ruin theory and queuing theory (see Asmussen and Albrecher \cite{AsumussenAlbrecher2010}).  We say that $F$ is {\it non-lattice} if the support of $F$ is not on any lattice.
\begin{theorem}
\label{theoremsupremum}
 Suppose that  $-\infty < \E [X_1] < 0$ and that $G$ is non-lattice. Then the following assertions are equivalent.\\
$(\mathrm{i})$ $\pi \in \cals_\Delta$,\quad $(\mathrm{ii})$ $B^{-1}\nu  \in \cals_\Delta$, \quad 
$(\mathrm{iii})$ $B^{-1}\nu  \in \call_\Delta$ and $\pi(x+\Delta)\sim \nu(x+\Delta)$, \\
$(\mathrm{iv})$ $G_0 \in \cals_\Delta$, 
\quad $(\mathrm{v})$ $G_0  \in \call_\Delta$ and $ \pi(x+\Delta)\sim \lambda/(1-\lambda) G_0(x+\Delta)$. 
\end{theorem}

\subsection{Randomly stopped sums}
Let $\{X_n\}_{n=1}^\infty$ and $S_n,\,n\ge 0$ be respectively the random sequence and 
its partial sum as defined above. 
Independently we take a counting r.v. $\tau$ on $\Z_+$. Then the distribution $S_\tau$ is given by 
\begin{align}
\label{def:dist:random:sums}
 F^{\ast \tau} := \sum_{n=0}^\infty \P(\tau=n) F^{\ast n},
\end{align}
where $F^{\ast 0}=\delta_0$. The following is a two-sided extension of \cite[Theorem 3.1]{Foss:Korshunov:Zachary:2013}. 
\begin{theorem}
\label{thm:random:sums}
 Suppose $F\in \call_\Delta$ and $\E [\tau] <\infty$. \\
$(\mathrm{i})$ If $F\in \cals_\Delta \cap \cald_\Delta$ and $\E[(1+\delta)^{\tau}]<\infty$ for some $\delta>0$, then 
\begin{align}
\label{eq:random:sums}
\P(S_\tau \in x+\Delta) /F(x+\Delta) \to \E[\tau]\quad \text{as}\quad x\to\infty. 
\end{align}
$(\mathrm{ii})$ If $\P(\tau>1)>0$ and further the relation \eqref{eq:random:sums} holds, then $F\in \cals_\Delta$. 
\end{theorem}
The proof is quite similar to that of \cite[Theorem 3.1]{Foss:Korshunov:Zachary:2013} 
if one uses the two-sided Kesten bound (Proposition \ref{prop:kesten}), and thus it is omitted.

\section{Proofs}
Let $c$ be a generic positive constant whose value is not of interest. Notice that 
we abuse this $c$ also for the definition of $\Delta$ and we make cautions where the double use  
may possibly cause confusion.

%\begin{proof}[Proof of Lemma \ref{lem:ald:symp:equiv}]
%Since $\alpha$ is al.d. and $\alpha(x)\sim \beta(x)$, for any $\vep>0$ there exists 
%$x_0,\,K>0$ such that for all $x>x_0$ and $y>0$
%\[
%1-\vep \le \alpha(x)/\beta(x) \le 1+\vep, \quad \text{and}\quad 
% \alpha(x+y)\le K \alpha(x), 
%\]
%so that 
%\[
% \beta(x+y) \le \alpha(x+y)/(1-\vep) \le K/(1-\vep)\alpha(x)\le (1+\vep)/(1-\vep)K \beta(x)
%\]
%holds for all $x>x_0$ and $y>0$. 
%\end{proof}

\begin{proof}[Proof of Lemma \ref{lem:equiv:ani:function}]
Obviously the a.n.i. properties imply \eqref{lem:ani:function}, and we show the 
converse. We observe that for sufficiently large $x>0$, 
\[
 1 \le \frac{\sup_{t\ge x}\alpha(t)}{\alpha(x)} = \sup_{t\ge x} \frac{\alpha(t)}{\beta(x)}\frac{\beta(x)}{\alpha(x)} 
\le \sup_{t\ge x}\frac{\alpha(t)}{\beta(t)} \frac{\beta(x)}{\alpha(x)} \to 1
\]
as $x\to\infty$. Moreover, since $\alpha(x)$ is positive on $[x_0,\infty)$ for some $x_0>0$, and 
$\alpha(x) \to 0$ as $x\to\infty$, there exists $y_x \le x$ such that 
\[
 1 \ge \frac{\inf_{x_0 \le t \le x}\alpha(t)}{\alpha(x)} = \inf_{y_x\le t \le x} \frac{\alpha(t)}{\beta(x)}\frac{\beta(x)}{\alpha(x)} 
\ge \inf_{y_0 \le t \le x} \frac{\alpha(t)}{\beta(t)} \frac{\beta(x)}{\alpha(x)},
\]
where $y_0\in (x_0,y_x]$ is arbitrary. 
Letting $x\to\infty$, so that one may let $y_x\to \infty$ and then $y_0 \to\infty$, the right converges to $1$ and we obtain the result.
It is known that $f\in\call$ implies that $f(x)\to 0$ as $x\to \infty$ (see \cite[p.76]{Foss:Korshunov:Zachary:2013}). 
Next we see that $F(x+\Delta)\to 0$ as $x\to \infty$. %For any $x>0$ we have
But this follows from  
\[
 \sum_{n=-\infty}^{\infty} F(x+nc+\Delta) =1,\quad x>0
\]
together with the Cauchy criterion for summability. 
The final result is trivial if we take arbitrary $\Delta$ in the previous argument. 
%Then 
%\[
% 0\le F(y+\Delta) \le \sum_{n:x+(n+1)c \ge y}^\infty F(x+nc+\Delta) \to 0
%\]
%as $y\to\infty$ follows by the Cauchy criterion for integrability. 
\end{proof}
\begin{proof}[Proof of Lemma \ref{lem:ldel+ani}]
Take a sufficiently large $x_0>0$ such that $\inf_{t\in [x_0,x]} f(t)\sim f(x)$. 
For any $y\in \R$ 
\begin{align*}
 \frac{f(x+y)}{f(x)} &= \frac{\sup_{t\ge x+y}f(t)}{\inf_{s\in [x_0,x]}f(s)} \frac{f(x+y)}{\sup_{t\ge x+y}f(t)}
\frac{\inf_{s\in [x_0,x]}f(s)}{f(x)} \\
& \ge \frac{F(x+y+\Delta)}{F(x-c+\Delta)} \frac{f(x+y)}{\sup_{t\ge x+y}f(t)}
\frac{\inf_{s\in [x_0,x]}f(s)}{f(x)}. 
\end{align*}
Moreover, 
\begin{align*}
 \frac{f(x+y)}{f(x)} &= \frac{\inf_{t\in [x_0,x+y]}f(t)}{\sup_{s\ge x}f(s)} \frac{f(x+y)}{\inf_{t\in [x_0,x+y]}f(t)}
\frac{\sup_{s\ge x}f(s)}{f(x)} \\
& \le \frac{F(x+y-c+\Delta)}{F(x+\Delta)}\frac{f(x+y)}{\inf_{t\in [x_0,x+y]}f(t)}
\frac{\sup_{s\ge x}f(s)}{f(x)}. 
\end{align*}
Now since $F \in \call_{\Delta}$ and $f \in \cala$, by taking $\lim_{x\to\infty}$ in both inequalities, we have 
\[
 1=\lim_{x\to\infty} \frac{F(x+y+\Delta)}{F(x-c+\Delta)} \le \lim_{x\to\infty}  \frac{f(x+y)}{f(x)} \le 
\lim_{x\to\infty}  \frac{F(x+y-c+\Delta)}{F(x+\Delta)} =1. 
\] 
Thus we have $f\in \call$. 
\end{proof}

\begin{proof}[Proof of Proposition \ref{prop:sdel+ani}]
Since the assertion of $\cals_{loc}$ follows from that of $\cals_{\Delta}$, we only prove the result for $\cals_{\Delta}$. 
It suffices to see that $F_+\in \cals_\Delta \Rightarrow F \in \cals_\Delta$. Although 
we can use $F_+(x+\Delta)\asymp F^+(x+\Delta)$ together with \cite[Lemma 4.20]{Foss:Korshunov:Zachary:2013} for the proof, 
we give a direct proof for the latter use. Notice that the proof of  \cite[Lemma 4.20]{Foss:Korshunov:Zachary:2013} was omitted 
by similarity. 
%\[
% F^+ \in \cals_\Delta \Rightarrow F \in \cals_\Delta\quad \text{and}\quad  F_+ \in \cals_\Delta \Rightarrow F \in \cals_\Delta.
%\] 
%The first part follows immediately from \cite[Lemma 4.27]{Foss:Korshunov:Zachary:2013}. }

%Since $F\in \cala_\Delta$, by Lemma \ref{lem:equiv:ani:function} there exists an non-increasing function $\alpha$ and for any $\vep>0$ we can 
%take $x_0>0$ such that for all $x_0>0$ 
%\[
% 1-\vep \le F(x+\Delta)/\alpha(x) \le 1+\vep.
%\]
%Thus for $y>0$ 
%\[
% F(x+y+\Delta) \le (1+\vep) \alpha(x+y) \le (1+\vep) \alpha(x) \le \frac{1+\vep}{1-\vep} F(x+\Delta) 
%\]
%holds for all $x>x_0$. Now by \cite[Lemma 4.27]{Foss:Korshunov:Zachary:2013}, $F\in \cals_\Delta$ follows. 

We observe that 
\[
 F_+(x) = (F(x)-F(0-))/\ov F(0-), 
\]
where $F(x-)=\lim_{y\uparrow x}F(y)$, so that $\ov F(0-)=1-F(0-)$ and 
\begin{align}
\label{eq:subexpo:F+}
 F_+^{\ast 2}(x+\Delta) &= \frac{1}{\ov F(0-)^2} \Big(
\int_{[0,x]} F(x+\Delta -y) F(dy) \\
&\hspace{2.5cm} +\int_{(x,x+c]} (F(x+c-y)-F(0-)) F(dy)
\Big). \nonumber  
\end{align}
The condition $F_+ \in \cals_\Delta$ implies that 
\begin{align}
\label{eq:subexpo:F+:F}
 \frac{F_+^{\ast 2}(x+\Delta)}{F_+(x+\Delta)} 
= \frac{F_+^{\ast 2}(x+\Delta) \ov F(0-) }{F(x+\Delta)} 
%= 
%\frac{1}{\ov F(0-) F(x+\Delta)} \Big(
%\int_{[0,x]} F(x+\Delta-y) F(dy) + \int_x^{x+c} (F(x+c-y)-F(0))F(dy)
%\Big) 
\to 2
\end{align}
as $x\to\infty$. Noticing \eqref{eq:subexpo:F+} we write 
\begin{align*}
  \frac{F^{\ast 2}(x+\Delta)}{F(x+\Delta)} & = \int_{(-\infty,\infty)}
 \frac{F(x+\Delta-y)}{F(x+\Delta)} F(dy) \\
& = \frac{F_+^{\ast 2}(x+\Delta) \ov F(0-)^2}{F(x+\Delta)} +2 \int_{(-\infty,0)} \frac{F(x+\Delta -y)}{F(x+\Delta)} F(dy), 
%& = \frac{1}{F(x+\Delta)} \Big(\big(
%\int_{[0,x]} + \int_{(-\infty,0)} \big)F(x+\Delta-y)F(dy) %+ \int_{(-\infty,0)} F(x+\Delta -y) F(dy) 
%+ \int_{(x,x+c]} F(x+c-y) F(dy) \\
%& \hspace{3cm} + \int_{(x+c,\infty)} F(x+c-y) F(dy) - \int_{(x,\infty)} F(x-y) F(dy)
%\Big)
\end{align*}
where we exploit the identity by Fubini:  
\begin{align}
\begin{split}
\label{formula:fubini}
 \int_{(x,\infty)} F(x-y)F(dy) &= 
\int_{\R^2} {\bf 1}_{\{x< y\le x-s\}}(F\times F)(d(s,y)) \\
%= \int_{(-\infty,\infty)\times (x,\infty)} {\bf 1}_{\{s+y\le x\}}(F\times F)(d(s,y)) \\
%&= \int_{(-\infty,\infty)} (F(x-s)-F(x)){\bf 1}_{\{s<0\}} F(ds) \\
&= \int_{(-\infty,0)} F(x-s)F(ds) -F(x)F(0-). 
\end{split}
\end{align}
%Thus we obtain from \eqref{eq:subexpo:F+} that  %Moreover, due to the integral by parts, we have 
%\begin{align*}
%  \frac{F^{\ast 2}(x+\Delta)}{F(x+\Delta)} 
%&= 
%\frac{1}{F(x+\Delta)} \Big(
%\int_{[0,x]}
% F(x+\Delta-y)  F(dy) + \int_{(x,x+c]}(F(x+c-y)-F(0-)) F(dy) \Big) \\
%& \quad + 2 \int_{(-\infty,0)} \frac{F(x+\Delta -y)}{F(x+\Delta)} F(dy) \\
%& = \frac{F_+^{\ast 2}(x+\Delta) \ov F(0-)^2}{F(x+\Delta)} +2 \int_{(-\infty,0)} \frac{F(x+\Delta -y)}{F(x+\Delta)} F(dy).
%&= \ov F(0) F_+^{\ast 2}(x+\Delta)  + 2\int_{-\infty}^0 
%\frac{F(x+\Delta-y)}{F(x+\Delta)} F(dy). 
%\end{align*}
Since we have \eqref{eq:subexpo:F+:F}, it suffices to show that the second integral term 
converges to $2 F(0-)$. Since $F \in \call_\Delta$, 
there exists a function $\alpha$ such that $\alpha(x)<x/2$, 
$\alpha (x)\to\infty$ as $x\to\infty$ and 
$F(x+\Delta)$ is $\alpha$-insensitive. Then 
\[
 2\int_{(-\alpha(x),0)} \frac{F(x+\Delta-y)}{F(x+\Delta)} F(dy) \to 2 F(0-),\quad \text{as}\quad x\to\infty,
\]
while by $F\in \cald_\Delta$, 
%and moreover, due to $F\in \cala_\Delta$, 
\[
 2 \int_{(-\infty,-\alpha(x)]} \frac{F(x+\Delta-y)}{F(x+\Delta)} F(dy) \le  2 K \int_{(-\infty,-\alpha(x)]} F(dy) \to 0\quad \text{as}\quad x\to\infty. 
\]
Now we obtain the desired result. 
%Thus $F^{\ast 2}(x+\Delta) /F(x+\Delta) \to 2\ov F(0)+2F(0)=2$ as $x\to\infty$. 
\end{proof}

\begin{proof}[Proof of Proposition \ref{prop:sac+ani}]
In view of Proposition \ref{prop:sdel+ani} and $\cals_{\Delta} \supset \cals_{loc} \supset \cals_{ac}$ it suffices 
to show that $F_+\in \cals_{\Delta} \Rightarrow F \in \cals_{ac}$. We observe by Lemma 
\ref{lem:ldel+ani} that $f_+ \in \call$, so that $f_+^{\ast 2} \in \call$. Then the uniform convergence property of $\call$ yields 
\[
 F_+ (x+\Delta) \sim c f_+(x)\ \text{and}\ F_+^{\ast 2}(x+\Delta) \sim cf_+^{\ast 2}(x). 
\]
Hence, $f_+\in \cals$ if and only if $F_+\in \cals_{\Delta}$. Finally, by the a.n.i. property $f_+\in \cals$ implies $f \in \cals$. 
The remaining result is trivial if we put $F=F_+$. 
\end{proof}

%\begin{proof}[Proof of Lemma \ref{lem:lloc+ani}]
% By lemma \ref{lem:sloc+lac}, we only prove the first assertion. But the proof immedetly follows from Lemma \ref{lem:ldel+ani}.  
%\end{proof}
%\textcolor{red}{
%\begin{proof}[Proof of Lemma \ref{lem:ifandonlyif:sub:delta+}]
%In view of \cite[Theorem 4.2.1]{Foss:Korshunov:Zachary:2013} it suffices to see that 
%\eqref{ifandlonly:sub:delta+:int} implies \eqref{ifandonlyif:sub:delta+}. 
%We observe that 
%\begin{align*}
%& \P(\xi_1+\xi_2 \in x+\Delta,\,\xi_1>\alpha(x),\,\xi_2>\alpha(x)) \\
%& = \E\big[
%\E[
%{\bf 1}_{\{x-\xi_2<\xi_1\le x+c-\xi_2\}{\bf 1}_{\{\xi_1>\alpha(x)\}}{\bf 1}_{\{\xi_2>\alpha(x)\}} }\mid \xi_2
%]  \big] \\
%& =\E\big[
%\E[
%{\bf 1}_{\{\alpha(x)\vee (x-\xi_2) <\xi_1 \le x+c-\xi_2\}} \mid \xi_2
%] {\bf 1}_{\{\xi_2>\alpha(x)\}}
%\big] \\
%& = \int_{(\alpha(x),x-\alpha(x)]} F(x+\Delta -y)F(dy) +\int_{(x-\alpha(x),x+c-\alpha(x)]}\big(
%F(x+c-y)-F(\alpha(x))
%\big) F(dy).
%\end{align*}
%The second integral has the bound 
%\[
% F(\alpha(x)+\Delta)F(x+\Delta-\alpha(x)) =o(F(x+\Delta)),
%\]
%where we notice that $F(x+\Delta)$ is $\alpha$-insensitive and $F(\alpha(x)+\Delta)\to 0$ by the integrability of $F$. 
%Thus \eqref{ifandlonly:sub:delta+:int} implies \eqref{ifandonlyif:sub:delta+}. 
%\end{proof}}

\begin{proof}[Proof of Lemma \ref{lem:ifandonlyif:sub:delta}]
$(\mathrm{i})\Rightarrow (\mathrm{ii})$. %Suppose that $F\in \cals_\Delta$. 
Define the event $B=\{\xi_1+\xi_2 \in x+\Delta\}$. We have 
\begin{align}
 \P(B) %& \P\big(B,\,-\alpha(x) < \xi_1 \le \alpha(x)\big)+\P\big(B,\,-\alpha(x)< \xi_2 \le \alpha(x)\big) \nonumber \\
       %& + \P\big(B,\,\xi_1 \le -\alpha(x)\big)+\P\big(B,\,\xi_2 \le -\alpha(x)\big) + \P\big(B,\,\xi_1 > \alpha(x),\,\xi_2 > \alpha(x)\big) \nonumber \\
\label{eq:decomp:b} =& 2\P\big(B,\,-\alpha(x)< \xi_1 \le \alpha(x)\big) + 2 \P\big(B,\,\xi_1 \le -\alpha(x)\big)\\
&\quad +\P\big(B,\,\xi_1 > \alpha(x),\,\xi_2 > \alpha(x)\big). \nonumber 
\end{align}
By Fatou's lemma 
\begin{align}
\label{liminf:b}
 \liminf_{x\to\infty} \frac{\P\big(B,\,-\alpha(x) < \xi_1\le \alpha(x)\big)}{F(x+\Delta)} 
= \liminf_{x\to \infty} \int_{(-\alpha(x),\alpha(x)]} \frac{F(x+\Delta-y)}{F(x+\Delta)} F(dy) \ge 1.
\end{align}
Hence from \eqref{eq:decomp:b} and \eqref{liminf:b} and the definition of $\cals_\Delta$, we obtain \eqref{ifandonlyif:sub:delta:common}. \\
$(\mathrm{ii})\Rightarrow(\mathrm{iii})$. This part is trivial since the condition $F\in \call_\Delta$ implies the existence of a function $\alpha$ 
w.r.t. which $F(x+\Delta)$ is $\alpha$-insensitive. \\
$(\mathrm{iii})\Rightarrow (\mathrm{i})$. %Now suppose that that the condition $(\mathrm{iii})$ holds for some function $\alpha$. 
We again use the decomposition \eqref{eq:decomp:b} for $B$ as defined above. Then 
\begin{align*}
 \P\big(B,-\alpha(x)< \xi_1 \le \alpha(x)\big) &= \int_{(-\alpha(x),\alpha(x)]}F(x-y+\Delta)F(dy) \\
& \sim F(x+\Delta) \int_{(-\alpha(x),\alpha(x)]} F(dy) \\
& \sim F(x+\Delta)
\end{align*}
%as $x\to\infty$ 
and so \eqref{ifandonlyif:sub:delta:real} and \eqref{ifandonlyif:sub:delta:common} together with \eqref{eq:decomp:b} 
imply $\cals_\Delta$ of $F$. 

If $F$ is a distribution on $\R_+$, then $\P\big(B,\,\xi_1 \le -\alpha(x)\big)=0$ and \eqref{ifandonlyif:sub:delta:real} is 
automatically satisfied. For the last assertion, we observe that  
\begin{align*}
& \P(\xi_1+\xi_2 \in x+\Delta,\,\xi_1>\alpha(x),\,\xi_2>\alpha(x)) \\
%& = \E\big[
%\E[
%{\bf 1}_{\{x-\xi_2<\xi_1\le x+c-\xi_2\}{\bf 1}_{\{\xi_1>\alpha(x)\}}{\bf 1}_{\{\xi_2>\alpha(x)\}} }\mid \xi_2
%]  \big] \\
& =\E\big[
\E[
{\bf 1}_{\{\alpha(x)\vee (x-\xi_2) <\xi_1 \le x+c-\xi_2\}} \mid \xi_2
] {\bf 1}_{\{\xi_2>\alpha(x)\}}
\big] \\
& = \int_{(\alpha(x),x-\alpha(x)]} F(x+\Delta -y)F(dy) +\int_{(x-\alpha(x),x+c-\alpha(x)]}\big(
F(x+c-y)-F(\alpha(x))
\big) F(dy).
\end{align*}
The second integral has the bound 
\[
 F(\alpha(x)+\Delta)F(x+\Delta-\alpha(x)) =o(F(x+\Delta)),
\]
where we notice that $F(x+\Delta)$ is $\alpha$-insensitive and $F(\alpha(x)+\Delta)\to 0$ by the integrability of $F$. 
Thus \eqref{ifandlonly:sub:delta+:int} implies \eqref{ifandonlyif:sub:delta:common}.  
\end{proof}

\begin{proof}[Proof of Proposition \ref{prop:factrization:delta}]
Let $z\in \R$. For sufficiently large $d>0$ 
\begin{align*}
 H(x+\Delta+z) &\ge \int_{[z,x+z-d]} F(x+\Delta+z-y)G(dy) \\
&\ge \inf_{y \in [d,x]}F(y+\Delta) G([z,x+z-d]).
\end{align*} 
Since $F\in \cala_\Delta$, 
\begin{align*}
 \liminf_{x \to \infty} \frac{H(x+\Delta)}{F(x+\Delta)} 
&= \lim_{z\to -\infty} \liminf_{x\to\infty} \frac{H(x+\Delta)}{H(x+\Delta+z)} 
\frac{H(x+\Delta+z)}{\inf_{y\in [d,x]} F(y+\Delta)} \frac{\inf_{y\in [d,x]}F(y+\Delta)}{F(x+\Delta)} \\
& \ge \lim_{z\to-\infty}G([z,\infty))=1, 
\end{align*}
which implies 
\begin{align}
\label{pf:limsup}
 \limsup_{x \to\infty} \frac{F(x+\Delta)}{H(x+\Delta)}\le 1. 
\end{align}
We will prove 
\begin{align}
\label{pf:liminf}
 \liminf_{x\to\infty} \frac{F(x+\Delta)}{H(x+\Delta)}\ge 1. 
\end{align}
Let $\xi_1$ and $\xi_2$ be r.v.'s whose distribution functions are respectively given by $G$ and $F$, and define the event 
$B=\{\xi_1+\xi_2 \in x+\Delta\}$. 
We take a function $\alpha(x)$ such that $\alpha(x)\to \infty$, $0<\alpha(x)<x/2$ and $H$ is $\alpha$-insensitive.
We study 
\begin{align*}
 \P(B) &= \P\big(B,\,\xi_1\le \alpha(x)\big) +\P\big(B,\,\xi_2 \le \alpha(x)\big) +\P\big(B,\, \xi_1>\alpha(x),\,\xi_2>\alpha(x)\big). %\\
%&=:I_1(x)+I_2(x)+I_3(x). 
\end{align*}
Write 
\begin{align*}
 \P(B,\,\xi_1\le \alpha(x)) &= \Big( \int_{(-\infty,-\alpha(x))} +\int_{[-\alpha(x),\alpha(x)]} \Big) F(x+\Delta-y)G(dy) \\
&=: I_{1}(x)+I_{2}(x).
\end{align*}
Since $F \in \cala_\Delta$ %,\,H\in \call_\Delta$ 
and \eqref{pf:limsup}, we have 
\begin{align}
\label{lim:I11/H}
 \frac{I_{1}(x)}{H(x+\Delta)} \le \frac{\sup_{y\le 0}F(x+\Delta-y)}{F(x+\Delta)} 
\frac{F(x+\Delta)}{H(x+\Delta)}G(-\alpha(x)) \to 0
%\frac{\sup_{y\le -\alpha(x)}F(x+\Delta-y)}{F(x+\Delta+\alpha(x))} 
%\frac{F(x+\Delta+\alpha(x))}{H(x+\Delta+\alpha(x))} \frac{H(x+\Delta+\alpha(x))}{H(x+\Delta)} G(-\alpha(x)) \to 0
\end{align}
as $x\to \infty$. For the second integral, write 
\begin{align*}
 \frac{I_{2}(x)}{H(x+\Delta)} \le \frac{H(x+\Delta-\alpha(x))}{H(x+\Delta)} \frac{F(x+\Delta-\alpha(x))}{H(x+\Delta-\alpha(x))}
\frac{\sup_{y\le \alpha(x)}F(x+\Delta-y)}{F(x+\Delta-\alpha(x))}G([-\alpha(x),\alpha(x)]). 
\end{align*}
The terms other than $F(x+\Delta-\alpha(x))/H(x+\Delta-\alpha(x))$ converge to $1$, and we obtain 
\begin{align}
\label{liminf:I12/H}
 \liminf_{x\to\infty} \frac{I_{2}(x)}{H(x+\Delta)} \le \liminf_{x\to\infty} \frac{F(x+\Delta)}{H(x+\Delta)}.
\end{align}
When $H \in \cald_\Delta$, we observe that 
\begin{align*}
 \frac{\P(B,\xi_2 \le \alpha(x))}{H(x+\Delta)} &= \int_{(-\infty,\alpha(x)]} \frac{G(x+\Delta-y)}{H(x+\Delta-y)}\frac{H(x+\Delta-y)}{H(x+\Delta)}F(dy) \\
& \le K\sup_{y\le \alpha(x)} \frac{G(x+\Delta-y)}{H(x+\Delta-y)}  %\frac{\sup_{y\le \alpha(x)}H(x+\Delta-y)}{H(x+\Delta-\alpha(x))}
\frac{H(x+\Delta-\alpha(x))}{H(x+\Delta)}F(\alpha(x)),  
\end{align*}
where $K>0$ is a constant of $\cald_\Delta$.  
Since $G(x+\Delta)=o(H(x+\Delta))$ and $H\in \call_{\Delta}$, we obtain
\begin{align}
\label{limsup:I1/H}
 \limsup_{x\to\infty} \P(B,\xi_2 \le \alpha(x))/H(x+\Delta)=0.
\end{align} 
When $G(x+\Delta)=o(F(x+\Delta))$, we notice that for sufficiently large $x>0$ 
\begin{align*}
\frac{\P(B,\xi_2 \le \alpha(x))}{H(x+\Delta)} & = \int_{(-\infty,\alpha(x)]} \frac{G(x+\Delta-y)}{F(x+\Delta-y)} \frac{F(x+\Delta-y)}{F(x+\Delta-\alpha(x))} 
%\frac{F(x+\Delta-\alpha(x))}{H(x+\Delta-\alpha(x))} \frac{H(x+\Delta-\alpha(x))}{H(x+\Delta)} 
\frac{F(x+\Delta-\alpha(x))}{H(x+\Delta)}
F(dy) \\
&\le c \sup_{y\le \alpha(x)} \frac{G(x+\Delta-y)}{F(x+\Delta-y)}F(\alpha(x)),
\end{align*}
where we notice that $F\in \cala_\Delta$, $H \in \call_\Delta$ and \eqref{pf:limsup}. Thus \eqref{limsup:I1/H} again holds in this case. 

Finally, write 
\begin{align*}
% I_3
& \P(B,\,\xi_1>\alpha(x),\,\xi_2 > \alpha(x)) \\
& = \E\big[
\E[
{\bf 1}_{\{\alpha(x)<\xi_1 \le x+c-\xi_2\}} -
{\bf 1}_{\{\alpha(x)<\xi_1 \le x-\xi_2\}} \mid \xi_2
] 
{\bf 1}_{\{\xi_2>\alpha(x)\}}
\big] \\
%&= \int_{(\alpha(x),x+c-\alpha(x)]} \big(G(x+c-y)-G(\alpha(x))\big)F(dy) - \int_{(\alpha(x),x-\alpha(x)]} \big(G(x-y)-G(\alpha(x))\big)F(dy) \\
&=\int_{(\alpha(x),x-\alpha(x)]} G(x+\Delta-y)F(dy) +
\int_{(x-\alpha(x),x+c-\alpha(x)]}\big(G(x+c-y)-G(\alpha(x))\big)F(dy) \\
& \le \int_{(\alpha(x),x-\alpha(x)]}G(x+\Delta-y)F(dy) + G(\alpha(x)+\Delta)F(x+\Delta-\alpha(x)). 
\end{align*}
Since $H\in \call_\Delta$, \eqref{pf:limsup} and $G$ is integrable, the last term in the last line is $o(H(x+\Delta))$.
For the integral in the last line, let $\xi_3$ be the r.v. with distribution $H$ independent of $(\xi_1,\xi_2)$,
and we use the inequalities,
\begin{align*}
& \int_{(\alpha(x),x-\alpha(x)]}H(x+\Delta-y) F(dy) \\
& \le \P(\xi_2+\xi_3 \in x+\Delta,\,\xi_2>\alpha(x),\xi_3>\alpha(x)) \\
& \le \int_{(\alpha(x),x-\alpha(x)]} F(x+\Delta-y) H(dy)+F(\alpha(x)+\Delta)H(x+\Delta-\alpha(x)).
\end{align*} 
Then the integral is bounded as 
\begin{align*}
& \int_{(\alpha(x),x-\alpha(x)]} \frac{G(x+\Delta-y)}{H(x+\Delta-y)} H(x+\Delta-y) F(dy) \\
%& \le \sup_{(\alpha(x),x-\alpha(x)]} \frac{G(x+\Delta -y)}{H(x+\Delta-y)} \int_{(\alpha(x),x-\alpha(x)]}H(x+\Delta-y)F(dy) \\
& \le \sup_{\alpha(x)<y \le x-\alpha(x)} \frac{G(x+\Delta -y)}{H(x+\Delta-y)} 
\Big(\int_{(\alpha(x),x-\alpha(x)]} H(x+\Delta-y) H(dy) \\
&\hspace{5.5cm}+ F(\alpha(x)+\Delta)H(x+\Delta-\alpha(x))
\Big),
\end{align*}
where in the final step we use the inequalities above and \eqref{pf:limsup}. 
%$G(x+\Delta)=o(H(x+\Delta))$, the last term is $o(H(x+\Delta))$. 
%For the first integral, we define an integer set $\N_{x-\alpha(x)}:=\{i:1\le i< (x-\alpha(x))/c+1\}$ and 
%intervals $\Delta_i:=(\alpha(x)+(i-1)c,\alpha(x)+ic].$ Using these definitions, we observe that 
%for sufficiently large $x$ 
%\begin{align*}
%& \int_{(\alpha(x),x-\alpha(x))}F(x+\Delta-y)G(dy) \\
%& \le \sum_{i\in \N_{x-\alpha(x)}} \int_{\Delta_i}F(x+\Delta-y)G(dy) \\
%& \le C \sum_{i\in \N_{x-\alpha(x)}} \int_{\Delta_i}H(x+\Delta-y)G(dy) \\
%& \le C \sum_{i\in \N_{x-\alpha(x)}} \sup_{y \in \Delta_i} H(x+\Delta-y) G\big(\alpha(x)+(i-1)c+\Delta\big) \\
%& \le C  \sup_{y\ge \alpha(x)} \frac{G(y+\Delta)}{H(y+\Delta)} \sum_{i\in \N_{x-\alpha(x)}} \int_{\Delta_i} 
%\sup_{z\in \Delta_i} H(x+\Delta-z) H(dy) \\
%& \le C' \sup_{y\ge \alpha(x)} \frac{G(y+\Delta)}{H(y+\Delta)} \sum_{i\in \N_{x-\alpha(x)}} \int_{\Delta_i} H(x+\Delta-y)H(dy) \\
%& \le C' \sup_{y\ge \alpha(x)} \frac{G(y+\Delta)}{H(y+\Delta)} \int_{(\alpha(x),x-\alpha(x))} H(x+\Delta-y)H(dy),
%\end{align*}
%where we use \eqref{pf:limsup} and the uniform convergence property of $H(x+\Delta+y)/H(x+\Delta)$ with $y$ on any compact intervals. 
%The last integral is bounded by $H^{\ast 2}(x+\Delta)$. 
The second quantity in the last line is $o(H(x+\Delta))$ by the integrability of $F$. 
Now since Lemma \ref{lem:ifandonlyif:sub:delta} $\mathrm(\mathrm{ii})$ together with $H\in \cals_\Delta$ 
implies \eqref{ifandlonly:sub:delta+:int}, we obtain $\P(B,\,\xi_1>\alpha(x),\,\xi_2 > \alpha(x))/H(x+\Delta)\to 0$ as $x \to\infty$. 
Now this together with \eqref{lim:I11/H}, \eqref{liminf:I12/H} and \eqref{limsup:I1/H} yields the desired result. 
\end{proof}

\begin{proof}[Proof of Proposition \ref{prop:convroot:sdel+ani}]
 Let $F\in \cala_\Delta$. Then, $F$ satisfies Assumptions A, B, and C of \cite[Theorem 4.1]{Watanabe:2021}, and  
 the result follows.
%Thus we see from \cite[Theorem 4.1]{Watanabe:2021} 
%the proposition holds.
\end{proof}

\begin{proof}[Proof of Proposition \ref{prop:convroot:sdel+ani:two-side}]
%Let $q :=F[0,\infty)$ and define $F_{0+}(dx)=q^{-1}{\bf 1}_{[0,\infty)}F(dx)$ if $q>0$ and $F_{0+}(dx)=0$ otherwise. 
%Moreover, define $F_-(dx)=(1-q)^{-1}{\bf 1}_{(-\infty,0)}F(dx)$ if $q<1$ and $F_-(dx)=0$ otherwise, so that $F(dx)=q F_{0+}(dx)+(1-q) F_-(dx)$. 
%We also write $F^+(dx)=q F_{0+}(dx)+(1-q) \delta_0(dx)$. 
%For the proof we use Lemma \ref{lem:vepkax} below, for which we need the following two results. 
Recall that $F_+$ is the conditional distribution on $\R_+$ and that 
$F^+(dx)= \ov F(0-) F_+(dx)+F(0-)\delta_0(dx)$. We further let $F_-$ be the conditional distribution 
on $(-\infty,0)$, so that $F(dx)=\ov F(0-) F_+(dx)+F(0-)F_-(dx)$.
%and $F^+(dx)= \ov F(0-) F_+(dx)+F(0-)\delta_0(x)$. 
Our strategy is to show 
$(F^+)^{\ast N}(x+\Delta)\sim F^{\ast N}(x+\Delta)$ and apply Propositions \ref{prop:convroot:sdel+ani} and \ref{prop:sdel+ani}. 
The key is Lemma \ref{lem:vepkax} below, which we will use several times.

We start by the following inequality. For an integer $k:1\le k\le N$, there exist $c_k>0$ and $a_k\in \R$ such that 
\begin{align}
\label{ineq:F+}
 F_{+}^{\ast k}(x+\Delta)\le c_k \ov F(0-) ^{-N}(F^{\ast N}(x+a_k+\Delta)+F^{\ast N}(x+a_k+c+\Delta)). 
\end{align}
If $F_{+}=0$, this is trivial and otherwise we put $F_{+}^{\ast(N-k)}(a_k+\Delta)=c_k^{-1}>0$ for some $a_k\in \R$. 
Then, \eqref{ineq:F+} follows from 
\[
 F_{+}^{\ast(N-k)}(a_k+\Delta) F_{+}^{\ast k}(x+\Delta)\le F_{+}^{\ast N}(x+a_k+2\Delta) \le \ov F(0-)^{-N}F^{\ast N}(x+a_k+2\Delta). 
\]
%Second, $F\in \cals_\Delta$ implies 
%\begin{align}
%\label{lim:x1x2>a} 
%\lim_{a\to\infty}\limsup_{x\to\infty} \frac{\P(X_1+X_2 \in x+\Delta,X_1>a,X_2>a)}{F(x+\Delta)}=0,  
%\end{align}
%where $X_1$ and $X_2$ are independent and identical with the distribution $F$. 
%The proof of \eqref{lim:x1x2>a} is similar to that of Lemma 4.7 of \cite{Foss:Korshunov:Zachary:2013} and we omit it. 
\begin{lemma}
\label{lem:vepkax}
 Suppose that $F^{\ast N}\in \cals_\Delta \cap \cald_\Delta$ and $F\in \cala_\Delta$. Then, for any integer $k:1\le k\le N$ 
we can take a constant $A>0$ and choose functions $\vep_{k}(A,x)\ge 0$ independent of $y$ such that for $x >k A$ and $y>0$ 
\[
 F_{+}^{\ast k}(x+\Delta)-F_{+}^{\ast k}(x+y+\Delta) \ge -\vep_k(A,x)
\]
and moreover, 
\[
 \lim_{A\to\infty} \limsup_{x\to\infty} \frac{\vep_{k}(A,x)}{F^{\ast n}(x+\Delta)}=0.
\]
\end{lemma}
\begin{proof}
The proof is by induction. 
For $k=1$ %by $F_+\in \call_\Delta \cap \cala_\Delta$ 
take sufficiently large $A>0$ such that for $x>A$ and $y>0$
\[
 F_{+}(x+\Delta)-F_{+}(x+y+\Delta)\ge -\vep(A)F_{+}(x+\Delta), 
\]
where $\vep(A)\ge 0$ and $\lim_{A\to \infty} \vep(A)=0$, which is possible by $F_+\in \cala_\Delta$. 
By \eqref{ineq:F+} and $F^{\ast N}\in \call_\Delta $
\[
 \limsup_{x\to\infty} \frac{F_{+}(x+\Delta)}{F^{\ast N}(x+\Delta)}<\infty. 
\]
Thus, the lemma holds with $k=1$. Assume that the assertions hold with $1\le k \le n \le N-1$ and we show 
that also with $k=n+1$.
Let $X_1,X_2,X_3$ and $X_4$ be independent r.v.'s with the distribution $F_{+},F_{+}^{\ast n},F_{+}^{\ast (N-1)}$
and $F_{+}^{\ast (N-n)}$, respectively. Let $\P(X_3+X_4\in a+\Delta)=C>0$ for some $a\in \R_+$. 
We evaluate $F_+^{\ast (n+1)}(x+\Delta)-F_+^{\ast (n+1)}(x+y+\Delta),\,x>(n+1)A,\ y>0,$
and obtain a convenient representation. 
We use the decomposition 
\begin{align*}
 F_+^{\ast (n+1)}(x+c) %&= \big(
%\int_{[0,A]}+ \int_{(A,x+c-A]}+ \int_{(x+c-A,x+c]} \big) F_+^{\ast n}(x+c-u) F_+(du) \\
&=  \big(
\int_{[0,A]}+ \int_{(A,x+c-A]} \big) F_+^{\ast n}(x+c-u) F_+(du) \\
&\quad + \int_{[0,A]} (F_+(x+c-u)-F_+(x+c-A))F_+^{\ast n}(dy), 
\end{align*}
where we use Fubini as in \eqref{formula:fubini}, and derive 
\begin{align}
\label{f_+:n+1:delta}
 F_+^{\ast (n+1)}(x+\Delta) &= \int_{[0,A]} F_+^{\ast n}(x+\Delta-u)F_+(du) + \int_{[0,A]} 
F_+(x+\Delta-u)F_+^{\ast n}(du) \\
&\quad + \int_{(A,x-A]} F_+^{\ast n}(x+\Delta-u) F_+(du) \nonumber \\
&\quad + \int_{(x-A,x+c-A]}(F_+^{\ast n}(x+c-u)-F_+^{\ast n}(A)) F_+(du), \nonumber 
\end{align}
where the last two terms are equal to 
\begin{align}
\label{pf:identity}
 \P(X_1+X_2 \in x+\Delta,\,X_1>A,X_2>A)
\end{align}
by the same calculation as for $\P(B,\,\xi_1>\alpha(x),\,\xi_2 >\alpha(x))$ in the proof of Proposition \ref{prop:factrization:delta}. 
At there $(\alpha(x),G,F)$ are replaced with $(A,F_+^{\ast n},F_+)$.

Now applying  \eqref{f_+:n+1:delta} for $x>(n+1)A$ we have 
\begin{align*}
 & F_{+}^{\ast(n+1)}(x+\Delta)-F_{+}^{\ast (n+1)}(x+y+\Delta) \\
 & = \int_{[0,A]}(F_{+}^{\ast n}(x-u+\Delta)-F_{+}^{\ast n}(x-u+y+\Delta))F_{+}(du) \\
 &\quad + \int_{[0,A]}(F_{+}(x-u+\Delta)-F_{+}(x-u+y+\Delta))F_{+}^{\ast n}(du) \\
 &\quad + \P(X_1+X_2 \in x+\Delta, X_1>A,X_2>A)-\P(X_1+X_2 \in x+y+\Delta,X_1>A,X_2>A). 
\end{align*}
Here, since $X_1+X_2$ and $X_3+X_4$ are independent, the last quantity is bounded as 
\begin{align*}
 & \P(X_1+X_2 \in x+y +\Delta,X_1 >A, X_2>A) \\
 & \le C^{-1}\P(X_1+X_2 \in x+y+\Delta,X_1>A,X_2>A, X_3+X_4 \in a+\Delta) \\
 & \le C^{-1}\P(X_1+X_3+X_2+X_4 \in x+a'+2\Delta ,X_1+X_3>A,X_2+X_4>A)(=:I_3), 
\end{align*}
where $a'=a+y$ and thus 
\begin{align*}
 & F_{+}^{\ast (n+1)}(x+\Delta)- F_{+}^{\ast (n+1)}(x+y+\Delta) \\
 & \ge - \int_{[0,A]} \vep_n(A,x-u)F_{+}(du)- \int_{[0,A]}\vep_1(A,x-u)F_{+}(du)-I_3 \\
 &=: -I_1-I_2-I_3. 
\end{align*}
Notice that $F^{\ast N}(x-u+\Delta)/F^{\ast N}(x+\Delta)$ converges to $1$ uniformly in 
$u\in [0,A]$ as $x\to \infty$, and thus by the induction hypothesis 
\[
 \lim_{A\to \infty} \limsup_{x\to \infty} \frac{I_1}{F^{\ast N}(x+\Delta)} \le  
\lim_{A\to \infty} \limsup_{y\to \infty} \frac{\vep_n(A,y)}{F^{\ast N}(y+\Delta)}=0 
\]
and 
\[
 \lim_{A\to\infty} \limsup_{x\to\infty} \frac{I_2}{F^{\ast N}(x+\Delta)} \le 
\lim_{A\to \infty} \limsup_{y\to \infty} \frac{\vep_1(A,y)}{F^{\ast N}(y+\Delta)}=0. 
\]
Let $Y_1$ and $Y_2$ be independent and identical r.v.'s with the distribution $F^{\ast N}$.
By the similar identity as for \eqref{pf:identity} and $ F_{+}^{\ast N}(dx) \le  \ov F(0-) ^{-N}F^{\ast N}(dx)$ , we obtain 
\begin{align*}
 \frac{I_3}{F^{\ast N}(x+a'+\Delta)} &\le \ov F(0-) ^{-2N} \int_{[A,x+a'-A]} \frac{F^{\ast N}(x+a'+\Delta-u)}{F^{\ast N} (x+a'+\Delta)}F^{\ast N}(du) \\
&\quad + \ov F(0-) ^{-2N}
\int_{[A,x+a'-A]} \frac{F^{\ast N}(x+a'+c+\Delta-u)}{F^{\ast N} (x+a'+\Delta)}F^{\ast N}(du) \\
&\quad +  \ov F(0-) ^{-2N}\frac{F^{\ast N}(A+2\Delta) F^{\ast N}(x+a'+2\Delta-A)}{F^{\ast N}(x+a'+\Delta)}.
\end{align*}
%\begin{align*}
% I_3 & \le \int^{x+a'+2c-A}_{A+} F_{0+}^{\ast N}(x+a'+2\Delta-u) F_{0+}^{\ast N}(du) \\
%&\le q^{-2N} \int_{A+}^{x+a'+2c-A} F^{\ast N}(x+a'+2\Delta -u) F^{\ast N}(du) \\
%& \le q^{-2N} \P(Y_1+Y_2\in x+a'+2\Delta,Y_1>A,Y_2>A-2c).
%\end{align*}
Since $F^{\ast N}\in \cals_{\Delta}$, by 
%\[
% \lim_{A\to\infty} \limsup_{x\to\infty} \frac{\P(Y_1+Y_2\in x+2\Delta,Y_1>A,Y_2>A)}{F^{\ast N}(x+_\Delta)}=0. 
%\]
 \eqref{ifandonlyif:sub:delta:common} of Lemma \ref{lem:ifandonlyif:sub:delta} the two integrals converge to $0$ as 
$x\to \infty$ and then $A\to\infty$ uniformly in $a'>a$. By the integrability condition, the last quantity also converges to $0$ 
as $x\to\infty$ and then $A\to\infty$ uniformly in $a'>a$. Thus by $F^{\ast N} \in \cald_{\Delta}$ $I_3$ is bounded by  $I_3'$ independent of $y>0$ and
\[
 \lim_{A\to\infty} \limsup_{x\to\infty} \frac{I_3'}{F^{\ast N}(x+\Delta)}=0.
\]
Now by letting $\vep_{n+1}(A,x)=I_1+I_2+I_3'$, we obtain the assertion for $k=n+1$.
\end{proof}

Now we proceed to the proof and for convenience we write $q=\ov F(0-)$. 
For $x>NA$ we observe that 
\begin{align*}
& (F^+)^{\ast N}(x+\Delta) - F^{\ast N}(x+\Delta) \\
& = \sum_{k=1}^{N-1}\binom{N}{k} q^k (1-q)^{N-k} \int_{(-\infty,0)} 
\big( F_{+}^{\ast k}(x+\Delta) -F_{+}^{\ast k}(x-y+\Delta)
\big)  F_-^{\ast(N-k)}(dy) \\
& \ge - \sum_{k=1}^{N-1}\binom{N}{k} q^k (1-q)^{N-k}\vep_k(A,x). 
\end{align*}
Thus, we see that 
\begin{align}
\label{liminf:F+:FN}
 \liminf_{x\to\infty} \frac{(F^+)^{\ast N}(x+\Delta)}{F^{\ast N}(x+\Delta)} 
\ge 1-\sum_{k=1}^{N-1}\binom{N}{k} q^k(1-q)^{N-k} \lim_{A\to\infty} \limsup_{x\to\infty} \frac
{\vep_k(A,x)}{F^{\ast N}(x+\Delta)}=1. 
\end{align}
Moreover, for $x>(N+1)A$ 
\begin{align*}
& (F^+)^{\ast N}(x+\Delta)-F^{\ast N}(x-A+\Delta) \\
& = \sum_{k=1}^N \binom{N}{k} q^k (1-q)^{N-k} \int_{(-\infty,0)}
\big(F_{+}^{\ast k} (x+\Delta)-F_{+}^{\ast k}(x-A-y+\Delta)\big) F_-^{\ast (N-k)}(dy) \\
& \le J_1+J_2, 
\end{align*}
where 
\[
 J_1=\sum_{k=1}^{N-1} \binom{N}{k} q^k (1-q)^{N-k} \int_{(-A,0)} \vep_k(A,x-A-y) F_-^{\ast (N-k)}(dy) + q^N \vep_{N}(A,x-A)
\]
and 
\[
 J_2= \sum_{k=1}^{N-1} 
\binom{N}{k} q^k (1-q)^{N-k} \int_{(-\infty,-A]} F_{+}^{\ast k}(x+\Delta) F_-^{\ast (N-k)}(dy).
\]
Note that $F^{\ast N} \in \cals_\Delta$ and that $F^{\ast N}(x-A+\Delta)/F^{\ast N}(x+\Delta)$ converges to $1$ 
as $x\to\infty$. We see from Lemma \ref{lem:vepkax} that 
\begin{align}
\label{lim:limsup:J1}
 \lim_{A\to\infty} \limsup_{x\to\infty} \frac{J_1}{F^{\ast N}(x+\Delta)} =0.
\end{align}
We obtain from \eqref{ineq:F+} and $F^{\ast N} \in \call_\Delta$ that 
\begin{align}
\label{lim:limsup:J2}
 \lim_{A\to\infty} \limsup_{x\to\infty} \frac{J_2}{F^{\ast N}(x+\Delta)}=0. 
\end{align}
%Thus we have 
Now in view of \eqref{liminf:F+:FN}, \eqref{lim:limsup:J1} and \eqref{lim:limsup:J2} we have 
\[
 \limsup_{x\to \infty} \frac{(F^+)^{\ast N}(x+\Delta)}{F^{\ast N}(x+\Delta)} \le \lim_{x\to\infty} 
\frac{F^{\ast N} (x-A+\Delta)}{F^{\ast N}(x+\Delta)} + \lim_{A\to\infty} \limsup_{x\to\infty} \frac{J_1+J_2}{F^{\ast N}(x+\Delta)}=1
\]
and thus $(F^+)^{\ast N}(x+\Delta)\sim F^{\ast N}(x+\Delta)$. 
By Proposition \ref{prop:asympt:equiv} this implies $(F^+)^{\ast N}\in \cals_\Delta$, 
so that $F^+ \in \cals_\Delta$ follows from Propositions \ref{prop:convroot:sdel+ani}. 
Since $F \in \cala_\Delta$ we see  
from Proposition \ref{prop:sdel+ani} that $F\in \cals_\Delta$. 
\end{proof}

\begin{proof}[Proof of Corollary \ref{prop:convroot:sden+ani}] 
 Let $F(dx):=f(x)dx$ and  $f\in \cala$. Then we have  $F\in \cala_\Delta$. 
Since $ f^{*N} \in  \cals\cap \cald$ implies $F^{*N} \in  \cals_{\Delta} \cap \cald_{\Delta}$, we see from Proposition \ref{prop:convroot:sdel+ani:two-side}
that $F \in  \cals_{\Delta}$. Hence we find from  Proposition \ref{prop:sac+ani} that  $f \in  \cals$.
\end{proof}

\begin{proof}[Proof of Proposition \ref{prop:asympt:equiv}]
Choose a function $\alpha(x)$ such that  $\alpha(x) < x/2$, %for all $x>0$,  
$\alpha(x) \to \infty$  as $x \to \infty$ and the function
 $G(x+\Delta)$ is $\alpha$-insensitive. Let $ \zeta_1, \zeta_2$ be independent identical r.v.'s %such that  $\zeta_1$ and $\zeta_2$  have 
with the distribution $G$. By Lemma \ref{lem:ifandonlyif:sub:delta}, it suffices to prove \eqref{ifandonlyif:sub:delta:real} 
and \eqref{ifandonlyif:sub:delta:common}  
with $\xi_1,\xi_2$ and $F$ there replaced with $\zeta_1,\zeta_2$ and $G$ here. 
%that
%\begin{align}
%\begin{split}
% \P\big(\zeta_1+\zeta_2 \in x+\Delta,\zeta_1 <-\alpha(x)\big) &= o(G(x+\Delta)),  \\
%\text{and}\quad  \P\big(\zeta_1+\zeta_2 \in x+\Delta,\,\zeta_1 > \alpha(x),\,\zeta_2>\alpha(x) \big) &= o(G(x+\Delta)).  
%\end{split}
%\end{align} 
The proof of \eqref{ifandonlyif:sub:delta:common} is exactly the same as that of \cite[Proposition 4.22]{Foss:Korshunov:Zachary:2013}
 and is omitted.
%(3.14) is same as that in the proof of Proposition 4.22 of \cite{Foss:Korshunov:Zachary:2013}  and is omitted.  
We prove \eqref{ifandonlyif:sub:delta:common}. 
By the assumption,  we have %for some $C_1, C_2 >0$,
\begin{align}
\begin{split}
\label{eq:pf:prop:convroot:sden+ani}
\P\big(\zeta_1+\zeta_2 \in x+\Delta,\zeta_1 \le -\alpha(x)\big)  &=\int_{(-\infty,-\alpha(x)]}G(x-y+\Delta)G(dy) \\
&\leq c \int_{(-\infty,-\alpha(x)]}F(x-y+\Delta)G(dy) \\
%&\leq  c\int_{-\infty}^{-\alpha(x)}F(x+\Delta)G(dy) \\
&= o(F(x+\Delta))= o(G(x+\Delta)).
\end{split}
\end{align} 
Thus we have proved the assertion.  
The proof of the final part is done by noticing that $F \in \cals_\Delta$ implies $F_+\in \cals_\Delta$ 
and applying \cite[Theorem 4.22]{Foss:Korshunov:Zachary:2013}. 
\end{proof}
   
\begin{proof}[Proof of Proposition \ref{prop:asympt:equiv2}]
Let $\zeta_1$ and $\zeta_2$ be independent r.v.'s with distributions $G_1$ and $G_2$, respectively.  
Let $\alpha$ be a function as in Lemma \ref{lem:ifandonlyif:sub:delta} and 
%such that  $\alpha(x) < x/2$ for all $x>0$,  $\alpha(x) \to \infty$  as $x \to \infty$ and the function
$F(x+\Delta)$ is $\alpha$-insensitive. Define the event $B=\{\zeta_1+\zeta_2 \in x+\Delta\}$. Then
\begin{align}
 \P(B)=& \P\big(B,\,-\alpha(x)\le \zeta_1 \le \alpha(x)\big)+\P\big(B,\,-\alpha(x)\le \zeta_2 \le \alpha(x)\big) \nonumber \\
       & + \P\big(B,\,\zeta_1 < -\alpha(x)\big)+\P\big(B,\,\zeta_2 < -\alpha(x)\big) + \P\big(B,\,\zeta_1 > \alpha(x),\,\zeta_2 > \alpha(x)\big). \nonumber 
\end{align}
By the dominated convergence theorem, we have
\begin{align*}
 & \lim_{x \to \infty}\frac{\P\big(B,\,-\alpha(x)\le \zeta_1 \le \alpha(x)\big)}{F(x+\Delta)} \\
   &= \int_{[-\alpha(x),\alpha(x)]} \lim_{x \to \infty}\frac{G_2(x-y +\Delta)}{ F(x-y+\Delta)}\frac{ F(x-y+\Delta)}{ F(x+\Delta)} G_1(dy)=c_2   
\end{align*}
and similarly 
\begin{align*}
 & \lim_{x \to \infty}\frac{\P\big(B,\,-\alpha(x)\le \zeta_2 \le \alpha(x)\big)}{F(x+\Delta)}=c_1.  %\\
 %  &= \int_{[-\alpha(x),\alpha(x)]} \lim_{x \to \infty}\frac{G_1(x-y +\Delta)}{ F(x-y+\Delta)}\frac{ F(x-y+\Delta)}{ F(x+\Delta)} G_2(dy)=c_1.   \nonumber 
\end{align*}
By the same argument as  in \eqref{eq:pf:prop:convroot:sden+ani}, we obtain that
$$\P\big(B,\,\zeta_i < -\alpha(x)\big)=o(F(x+\Delta))\quad i=1,2. 
%\text{
%and}\quad 
%\P\big(B,\,\zeta_2 < -\alpha(x)\big)=o(F(x+\Delta)).
$$
By the same argument as that in the proof of \cite[Proposition 4.22]{Foss:Korshunov:Zachary:2013}, we see that
$$ \P\big(B,\,\zeta_1 > \alpha(x),\,\zeta_2 > \alpha(x)\big)=o(F(x+\Delta)).$$
Thus, \eqref{aymp:eqiv:doubleG} holds. 
The final assertion follows from Proposition \ref{prop:asympt:equiv2}.
%Thus we have proved the proposition.
\end{proof}

%\begin{proof}[Proof of Proposition \ref{prop:asympt:equiv3}] 
%By induction, Proposition \ref{prop:asympt:equiv2} implies the proposition.
%\end{proof}

\begin{proof}[Proof of Proposition \ref{prop:kesten}] 
For a sequence of iid r.v.'s $\{\xi_n\}$ with distribution $F$, we put $S_n=\sum_{k=1}^n\xi_k$. 
For $x_0 \ge 0$ and $k \ge 1$ we write  
$$A_k:=A_k(x_0)=\sup_{x>x_0}\frac{F^{*k}(x+\Delta)}{F(x+\Delta)}.$$
Define the event $B=\{\xi_1+\xi_2 \in x+\Delta\}$. 
Let $\vep >0$ be any fixed constant. 

By Lemma \ref{lem:ifandonlyif:sub:delta},  there exist $x_0 >0$ and $C>0$ such that $x_0>2C$ and, for any $x > x_0$, 
$$\P(B,\xi_2 \le C)+\P(B,\xi_1>C,\xi_2>C) \le (1+\vep/2)F(x+\Delta).$$
Since we have 
$$\P(B,\xi_1 \le C) \le \P(B,\xi_2>x-C),$$
we observe that for any $x > x_0$,  
\begin{align*}
\P(B,\xi_2\le x-x_0)& \le \P(B,\xi_2\le x-C)  \\
&=\P(B)-\P(B,\xi_2>x-C)  \\
& \le \P(B)- \P(B,\xi_1 \le C) \\
&=\P(B,\xi_2 \le C)+\P(B,\xi_1>C,\xi_2>C), 
\end{align*}
so that we can choose $x_0 >0$ such that for any $x > x_0$, 
\begin{equation}
\label{ineq:pf:prop:kesten}
\P(B,\xi_2\le x-x_0) \le (1+\vep/2)F(x+\Delta).
\end{equation}
For any $n>1$ and $x > x_0$, we write 
\begin{align}
\P(S_n \in x+\Delta)&=\P(S_n \in x+\Delta,\xi_n\le x-x_0) +\P(S_n \in x+\Delta,\xi_n > x-x_0) \nonumber \\ 
&=: P_1(x)+P_2(x),\nonumber
\end{align}
where, by the choice \eqref{ineq:pf:prop:kesten} of $ x_0$ and by the definition of $A_{n-1}$,
\begin{align}
P_1(x)&=\int_{(-\infty,x-x_0]}\P(S_{n-1} \in x-y+\Delta)F(dy)  \nonumber \\
&\le A_{n-1}\int_{(-\infty,x-x_0]}F(x-y+\Delta)F(dy)  \nonumber \\
&= A_{n-1}\P(B,\xi_2\le x-x_0) \nonumber \\
&\le  A_{n-1}(1+\vep/2)F(x+\Delta).\label{ineq:p1}
\end{align}
Moreover, 
\begin{align}
P_2(x)&=\int_{(-\infty,x_0+c]}\P(\xi_n \in x-y+\Delta,\xi_n > x-x_0)\P(S_{n-1} \in dy)  \nonumber \\
& \le \sup_{-\infty<t \le x_0+c}F(x-t+\Delta). \nonumber 
\end{align}
Thus, if $x > 2x_0$, then
$$P_2(x) \le L_1F(x+\Delta),$$
where
$$ L_1:=\sup_{-\infty<t \le x_0+c,\, y > 2x_0}\frac{F(y-t+\Delta)}{F(y+\Delta)}.$$
If $x_0 <x \le 2x_0$, then $P_2(x) \le 1$ implies
$$\frac{P_2(x)}{F(x+\Delta)}\le \frac{1}{\inf_{x_0 <x \le 2x_0}F(x+\Delta)}=:L_2.$$
Since $F \in \call_{\Delta}\cap \cald_\Delta$, $L_1$ and $L_2$ are finite for $x_0$ sufficiently large. 
Put $L=\max(L_1,L_2)$ and then, for any $x > x_0$ 
\begin{equation}
\label{ineq:p2}
P_2(x) \le LF(x+\Delta).
\end{equation}
It follows from \eqref{ineq:p1} and \eqref{ineq:p2} that for $n>1$
$$A_n \le   A_{n-1}(1+\vep/2)+ L.$$
Now by induction, it follows that 
$$A_n \le  A_1(1+\vep/2)^{n-1} +L\sum_{k=0}^{n-2}(1+\vep/2)^k \le Ln(1+\vep/2)^{n-1},$$
which implies the result. 
\end{proof}

\begin{proof}[Proof of Proposition \ref{prop:decomp:delta}]
With the form $G$, we may write 
\begin{align}
\label{eq:pf:factgene}
 1=p_G F(x+\Delta)/H(x+\Delta) +(1-p_G) F \ast G_1(x+\Delta)/H(x+\Delta), 
\end{align}
and put 
\[
\ov C := \limsup_{x\to \infty}  F(x+\Delta)/H(x+\Delta)\quad \text{and}\quad \underline{C}:=\liminf_{x\to\infty}  F(x+\Delta)/H(x+\Delta), 
\]
which are well-defined since $H(x+\Delta)\ge p_G F(x+\Delta)$.

Assume the first condition ($H\in \cald_\Delta$). 
Take an insensitive function $\alpha$ for $H(x+\Delta)$ and consider 
\begin{align*}
 \frac{F \ast G_1(x+\Delta)}{H(x+\Delta)} &= \Big(
\int_{(-\infty,-\alpha(x))}+ \int_{[-\alpha(x),\alpha(x)]} +\int_{(\alpha(x),\infty) } 
\Big)  \frac{F(x+\Delta-y) }{H(x+\Delta)}G_1(dy) \\
& =: I_1(x)+I_2(x)+I_3(x). 
\end{align*}
By Fatou's lemma and $H(x+\Delta)\in\call$,  
\begin{align}
 \liminf_{x\to\infty} I_2(x) &= \liminf_{x\to\infty} \int_{[-\alpha(x),\alpha(x)]} 
\frac{F(x+\Delta-y)}{H(x+\Delta-y)} \frac{H(x+\Delta-y)}{H(x+\Delta)} G_1(dy) \label{ineq:pf:I2i} \\
& \ge \int_{\R} \liminf_{x \to\infty} \frac{F(x+\Delta-y)}{H(x+\Delta-y)}
{\bf 1}_{\{y\in[-\alpha(x),\alpha(x)]\}} G_1(dy) \ge \underline{C}, \nonumber
\end{align}
where we may take a continuous $\alpha(x)$ if needed. 
Again by Fatou's lemma 
\begin{align}
 \limsup_{x\to \infty} I_2(x) \le \int_{\R} \limsup_{x\to\infty}  \frac{F(x+\Delta-y)}{H(x+\Delta-y)}
{\bf 1}_{\{y\in[-\alpha(x),\alpha(x)]\}} G_1(dy)\le \ov C.
\label{ineq:pf:I2s}
\end{align}
Let $\zeta_1$ and $\zeta_2$ be independent r.v.'s with the distributions $F$ and $G_1$, respectively. Then we observe
\begin{align*}
 I_3(x) &=\frac{\P(\zeta_1+ \zeta_2 \in x+\Delta, \zeta_2 \ge\alpha(x)) }{H(x+\Delta)}  \le \frac{\P(\zeta_1+ \zeta_2 \in x+\Delta, \zeta_1 \le x+c-\alpha(x)) }{H(x+\Delta)}. 
\end{align*}
Since  $H(dx)\ge p_G F(dx)$ and  $G_1(x+\Delta)=o(H(x+\Delta))$ we have 
\begin{align}
 \limsup_{x\to\infty} I_3(x) &\le \limsup_{x\to\infty} \int_{(-\infty,x-\alpha(x)+c]}
\frac{G_1(x+\Delta-y)}{H(x+\Delta-y)} \frac{H(x+\Delta-y)}{H(x+\Delta)} p_G^{-1} H( dy)  \label{ineq:pf:I3} \\
& \le \limsup_{x \to \infty} \frac{G_1(x+\Delta)}{H(x+\Delta)}  c \limsup_{x\to \infty} \frac{H^{\ast 2}(x+\Delta)}{H(x+\Delta)} =0.  
\nonumber
\end{align}
By the al.d. property of $H(x+\Delta)$, it follows that 
\begin{align}
\limsup_{x\to\infty} I_1(x) &= \limsup_{x\to\infty} \int_{(-\infty,-\alpha(x))}
\frac{F(x+\Delta-y)}{H(x+\Delta-y)} \frac{H(x+\Delta-y)}{H(x+\Delta)}  G_1(dy)  \label{ineq:pf:I4} \\
& \le c \limsup_{x\to\infty} \int_{(-\infty,-\alpha(x))}  G_1(dy)  =0. \nonumber
\end{align}
Collecting 
\eqref{ineq:pf:I2i}-\eqref{ineq:pf:I4}, we obtain 
\begin{align}
 \liminf_{x\to\infty} \frac{F \ast G_1(x+\Delta)}{H(x+\Delta)} \ge \underline{C}\quad \text{and}\quad \limsup_{x\to\infty} 
\frac{F \ast G_1(x+\Delta)}{H(x+\Delta)} \le \ov{C}. 
\end{align}
Now taking $\liminf_{x\to\infty}$ and $\limsup_{x \to \infty}$ on both sides of \eqref{eq:pf:factgene}, we have 
\begin{align}
\label{ineq:pf:factgene}
 p_G \ov C +(1-p_G) \underline{C}\ \le  1\  \le\ p_G \underline{C} +(1-p_G) \ov C 
\end{align}
and thus 
\[
 0\le (1-2 p_G)(\ov C-\underline{C}).
\]
The assumption $p_G \in (2^{-1},1)$ implies $\ov C=\underline{C}$. Moreover from \eqref{ineq:pf:factgene}, $\ov C=\underline{C}=1.$ 
Then noticing $H(x+\Delta)\ge p_GF(x+\Delta)$ for all $x\in \R$, we have 
$F\in \cals_\Delta$ from Proposition \ref{prop:asympt:equiv}.  

For the proof under the second condition, we only indicate different points. 
The quantities $\ov C$ and $\underline{C}$ are the same. We work on another form 
\begin{align*}
 \frac{F\ast G_1(x+\Delta)}{H(x+\Delta)} &= 
\int_{[0,\alpha(x)]} \frac{F(x+\Delta-y)}{H(x+\Delta)} G_1(dy) \\
&\quad +\Big(
\int_{(\alpha(x),x+c]}
 \frac{F(x+c-y)}{H(x+\Delta)} G_1(dy) -\int_{(\alpha(x),x]} \frac{F(x-y)}{H(x+\Delta)} G_1(dy) \Big)\\
&=: J_1(x)+J_2(x). 
\end{align*}
Similarly as before we obtain 
\[
 \liminf_{x\to\infty} J_1(x)\ge \underline{C}\quad \text{and}\quad \limsup_{x\to\infty} J_1(x) \le \ov C. 
\]
For the $J_2(x)$, we use the following identity
\begin{align*}
 \int_{(\alpha(x),x]} F(x-y)G_1(dy) &= \int_{\R_+\times (\alpha(x),x]}{\bf 1}_{\{z \le x-y\}} (F\times G_1)(d(z,y)) \\
&= \int_{\R_+\times \R_+} {\bf 1}_{\{\alpha(x)<y \le x-z\}} {\bf 1}_{\{z< x-\alpha(x)\}} (F\times G_1)(d(z,y)) \\
& = \int_{ [0, x-\alpha(x)]} (G_1(x-z)-G_1(\alpha(x)) )F(dz), 
\end{align*} 
so that 
\begin{align*}
 J_2(x) &= \int_{[0,x-\alpha(x)]}\frac{G_1(x+\Delta-z)}{H(x+\Delta)} F(dz) + 
\int_{(x-\alpha(x),x+c -\alpha(x)]} \frac{G_1(x+c-z)-G_1(\alpha(x))}{H(x+\Delta)} F(dz) \\
&\le \int_{[0,x-\alpha(x)]}\frac{G_1(x+\Delta-z)}{H(x+\Delta-z)} \frac{H(x+\Delta-z)}{H(x+\Delta)} F(dz)
+ \frac{F(x+\Delta-\alpha(x))}{H(x+\Delta)} G_1(\alpha(x)+\Delta). 
\end{align*}
%$J_2(x)=\P(\xi_1+\xi_2 \in x+\Delta,\,\xi_2 >\alpha(x))$
%\begin{align*}
% H(x+\Delta) J_2(x) &= \P(\xi_1+\xi_2 \in x+\Delta ,\,\xi_2>\alpha(x)) \\
%&= \P((x-\xi_1)\vee \alpha(x) <\xi_2 \le x+c-\xi_1) \\
%&= \int_{[0,x-\alpha(x)]} G(x+\Delta-y)F(dy)\\
%&\quad +\int_{(x-\alpha(x),x+c-\alpha(x)]} \big(G(x+c-y)-G(\alpha(x))\big)
%F(dy) \\
%&\le \int_{[0,x-\alpha(x)]}G(x+\Delta-y)F(dy)+ G(\alpha(x)+\Delta)F(x+\Delta-\alpha(x)).
%\end{align*}
Since $H(dx)\ge p_G F(dx)$ and $G(\alpha(x)+\Delta)\to 0$ as $x\to \infty$, 
%$G(x+\Delta-y)=o(H(x+\Delta-y))$ on $y \in [0,x-\alpha(x)]$, 
\[
 J_2(x)\le o(1) \int_{[0,x-\alpha(x)]} \frac{H(x+\Delta-y)}{H(x+\Delta)}H(dy)+o(1)=o(1).  
%G(\alpha(x)+\Delta)\frac{F(x+\Delta-\alpha(x))}{H(x+\Delta)}.
\]
%Now $H\in \cals_\Delta$ and the integrability of $G$ imply 
%\[
% \limsup_{x\to\infty} J_2(x)=0.
%\]
The remaining proof is the same as before.
\end{proof}

\begin{proof}[Proof of Proposition \ref{prop:decomp:delta:Poisson}]
By similarity, we only give the proof under that $H \in \cald_\Delta$.
 We assume non-degeneracy for $G$, since otherwise the proof is obvious. 
Let $\lambda$ be the Poisson parameter of $G$. 
If $\lambda < \log 2$, then since $e^{-\lambda}>2^{-1}$ 
the result is immediate from Proposition 2.14.
If $\lambda \ge \log 2$, we take an integer $n$ such that $\lambda/n <\log 2$, and 
define a compound Poisson distribution 
\[
G_{1/n}(dx)= e^{-\lambda/n}\delta_0(dx)+(1-e^{-\lambda/n})G_{1/n,1}(dx)
\]
%with then let 
%\[
% \wt g_{1/2}(x)=e^{-\lambda/2} \delta(x)+(1-e^{-\lambda/2})g_{1/2}(x)
%\]
 such that $G= G_{1/n}^{\ast n}$. 
%, so that $\wt g(x)=\wt g_{1/2}\ast \wt g_{1/2}$. 
Notice that 
\[
G(x+\Delta) = \big(e^{-\lambda/n}\delta_0+ (1-e^{-\lambda/n})G_{1/n,1} \big)^{\ast n}(x+\Delta) \ge n e^{-\lambda(n-1)/n}(1-e^{-\lambda/n})G_{1/n,1}(x+\Delta) 
%e^{-\lambda}\delta(x)+ 2(1-e^{-\lambda/2})e^{-\lambda/2}g_{1/2}(x)+(1-e^{\lambda/2})^2 g_{1/2}\ast g_{1/2}(x),
\]
and thus $G(x+\Delta)=o(H(x+\Delta))$ implies $G_{1/n}(x+\Delta)=o(H(x+\Delta))$. Now we apply Proposition 2.14  to 
$G_{1/n}\ast (G_{1/n}^{\ast (n-1)}\ast F)$ with $G_{1/n}$ be the negligible part, 
and obtain that $G_{1/n}^{\ast (n-1)}\ast F \in \cals_\Delta$ and $H(x+\Delta)\sim G_{1/n}^{\ast (n-1)}\ast F(x+\Delta)$.
Since $H \in \cald_\Delta$, so is $G_{1/n}^{\ast (n-1)}\ast F$ by Lemma \ref{lem:ald:symp:equiv}.
We iterate this step until we reach $F\in \cals_\Delta$ and $H(x+\Delta)\sim F(x+\Delta) $.
\end{proof}

\section{Proofs for Section \ref{sec:main}}
Note that we abuse the notation $c$, which has two meanings.
One is for $c$ of $\Delta$ and the other is arbitrary positive constant whose values are not of interest. 
This makes no confusion. 

\begin{proof}[Proof of Theorem \ref{thm:cp:twosided}]
%Note that we abuse the notation $c$, which has two meanings.
%One is for $c$ of $\Delta$ and the other is arbitrary positive constant whose values are not of interest. 
%This makes no confusion. 
%\begin{align*}
%%\label{def:cp:distribution}
% \mu(dx)= e^{-\lambda} \delta_0(dx)+(1-e^{-\lambda})\mu_0(dx),
%\end{align*}
%where 
%\begin{align}
%\label{eq:cp:proper:distribution:part}
% \mu_0(dx)= (e^\lambda-1)^{-1} \sum_{n=1}^\infty (\lambda^n/n!)G^{\ast n}(dx). 
%\end{align}
We give the proof for the general two-sided case first, and then  
we indicate the difference for the positive-half case since the proof is quite similar. \\

\noindent
{\bf $[\,(\mathrm{i})$ implies $(\mathrm{ii})$ and $(\mathrm{iii})\,]$}\ 
We choose $c_1>0$ such that $\ov G(c_1)=\Lambda_1<\log 2/\lambda$ and define 
another compound Poisson by 
\begin{align}
\label{pf:def:cp:mu1}
 \wh \mu_1(z) =\exp \Big(
\lambda\Lambda_1 \int_{(c_1,\infty)} (e^{izx}-1)G_1(dy)
\Big),\quad G_1(x)=G(x)/\Lambda_1,
\end{align}
so that 
$
 \mu=\mu_1\ast \mu_2,\, \text{i.e.}\, \wh \mu(z)=\wh \mu_1(z) \wh \mu_2(z),
$
where 
\[
 \wh \mu_2(z)= \exp \Big(
\lambda\Lambda_2 \int_{(-\infty,c_1]} (e^{izx}-1)G(dy)/\Lambda_2
\Big),\quad \Lambda_2=G(c_1).
\]
Since $\mu_2$ has exponential moments of any orders (\cite[Theorem 25.3]{sato:1999}), we have for any $\gamma>0$  
\begin{align}
\label{deltasub:exp:moment}
 e^{\gamma x} \mu_2(x+\Delta) \le \int_x^{x+c} e^{\gamma y}\mu_2(dy) \to 0\quad \text{as}\ x\to\infty,
\end{align}
so that $\mu(x+\Delta)=o(e^{-\gamma x})$. Now by Proposition \ref{prop:decomp:delta:Poisson} with $(H,G,F)$ there replaced with 
$(\mu,\mu_2,\mu_1)$ here, we obtain 
\begin{align}
\label{relatoin:mu:mu1}
% \mu_1(x+\Delta) \in \cals
\mu_1\in \cals_\Delta \quad \text{and}\quad \mu(x+\Delta)\sim \mu_1(x+\Delta). 
\end{align}
This implies $\mu_1\in \cald_\Delta$ by Lemma \ref{lem:ald:symp:equiv}. 

We will see that $\mu_1\in \cals_\Delta\cap \cald_\Delta$ implies $G_1\in \cals_\Delta\cap \cald_\Delta$. 
Write 
\begin{align}
\label{pf:cp:def:mu1:mu10}
 \mu_1(dx)=e^{-\lambda\Lambda_1}\delta_0(dx)+(1-e^{-\lambda \Lambda_1})\mu_{10}(dx),
\end{align}
where 
\[
 \mu_{10}(dx)= (e^{\lambda \Lambda_1}-1)^{-1} \sum_{n=1}^\infty \big((\lambda \Lambda_1)^n/n! \big) G_1^{\ast n}(d x). 
\]
Consider the Laplace transform of $\mu_{10}$, 
\[
 L_{\mu_{10}}(z)= (e^{\lambda \Lambda_1 L_{G_1}(z)}-1)/(e^{\lambda\Lambda_1}-1),
\]
which yields,
\[
 \lambda \Lambda_1 L_{G_1}(z)= -\sum_{n=1}^{\infty}n^{-1}(1-e^{\lambda\Lambda_1})^n L_{\mu_{10}}^n(z). 
\]
Thus, by the inversion
\begin{align}
 \label{pf:cp:inversion:delta}
\lambda\Lambda_1 G_1(x+\Delta) = - \sum_{n=1}^\infty n^{-1}(1-e^{\lambda \Lambda_1})^n \mu_{10}^{\ast n}(x+\Delta). 
\end{align}
Take a sufficiently small $\vep>0$ such that $(e^{\lambda \Lambda_1}-1)(1+\vep)<1$. Since 
$\mu_{10}\in \cals_\Delta \cap \cald_\Delta$, by Proposition \ref{prop:kesten} there exists $C(\vep)$ such that 
$\mu_{10}^{\ast n}(x+\Delta)\le C(\vep)(1+\vep)^n \mu_{01}(x+\Delta)$ for $x$ sufficiently large. Applying 
the dominated convergence in \eqref{pf:cp:inversion:delta}, we obtain 
\begin{align}
\label{pf:asymp:g1:mu10}
 \lim_{x\to\infty} G_1(x+\Delta)/\mu_{10}(x+\Delta) = (1-e^{-\lambda \Lambda_1})/(\lambda \Lambda_1). 
\end{align}
Now correcting the results \eqref{pf:def:cp:mu1},\eqref{relatoin:mu:mu1},\eqref{pf:cp:def:mu1:mu10} and \eqref{pf:asymp:g1:mu10}, we see 
\[
 \lim_{x\to\infty} \frac{\mu(x+\Delta)}{G(x+\Delta)}=\lim_{x\to\infty}
\frac{\mu(x+\Delta)}{\mu_{10}(x+\Delta)}\frac{\mu_{10}(x+\Delta)}{G_1(x+\Delta)}\frac{G_1(x+\Delta)}{G(x+\Delta)}=\lambda. 
\]
By Proposition \ref{prop:asympt:equiv} and Lemma \ref{lem:ald:symp:equiv}, we obtain $G\in \cals_\Delta$ and $G\in \cald_\Delta$ respectively. 
Thus $(\mathrm{ii})$ and $(\mathrm{iii})$ are proved. \\

\noindent
{\bf $[\,(\mathrm{iii})$ implies $(\mathrm{ii})\,]$}\ 
%Notice that the second condition of $(\mathrm{iii})$ implies $G\in \call_\Delta$. 
In view of \eqref{def:cp:measure}, we write 
\[
 G^{\ast 2}(x+\Delta) =(2/\lambda^2) \big\{
%(e^\lambda-1)\mu_0(x+\Delta) 
e^{\lambda}\mu(x+\Delta)-\sum_{n\neq 2}^\infty (\lambda^n/n!)G^{\ast n}(x+\Delta)
\big\}. 
\]
Dividing this by $G(x+\Delta)$ and taking $\limsup_{x\to\infty}$ on both sides, we have by Fatou's lemma 
\begin{align*}
 \limsup_{x\to\infty} \frac{G^{\ast 2}(x+\Delta)}{G(x+\Delta)} &\le 
(2/\lambda^2) \Big\{
\lambda e^{\lambda} -\sum_{n\neq 2}^\infty (\lambda^n/n!)\liminf_{x\to\infty}\frac{G^{\ast n}(x+\Delta)}{G(x+\Delta)} 
\Big\} \\
&=(2/\lambda^2) \{
\lambda e^{\lambda}-\lambda \sum_{n\neq 1}^\infty \lambda^n/n!
\}=2,
\end{align*}
where we use \cite[Corollary 4.19]{Foss:Korshunov:Zachary:2013} together with $G \in \call_\Delta$. Again by \cite[Corollary 4.19]{Foss:Korshunov:Zachary:2013} with 
$n=2$, we obtain $G^{\ast 2}(x+\Delta)\sim 2 G(x+\Delta)$, which is $(\mathrm{ii})$. \\

\noindent
{\bf $[\,(\mathrm{ii})$ implies $(\mathrm{iii})$ and $(\mathrm{i})\,]$}\quad 
Notice that Proposition \ref{prop:kesten} holds with $G$. 
Thus applying $G^{\ast n}(x+\Delta)/G(x+\Delta)\to n$ as $x\to\infty$ and the 
dominated convergence to 
\[
 \frac{\mu(x+\Delta)}{G(x+\Delta)}= e^{-\lambda} \sum_{n=1}^\infty (\lambda^n/n!) \frac{G^{\ast n}(x+\Delta)}{G(x+\Delta)},
\]
we prove $(\mathrm{iii})$. Then a combination of $(\mathrm{ii})$ and $(\mathrm{iii})$ yields $\mu\in \cals_\Delta$ by Proposition \ref{prop:asympt:equiv}.
Here $\mu \in \cald_\Delta$ follows from that of $G$.

Next we proceed to the proof for the positive half case. As stated before we only indicate the difference. 
For the part [{\bf $(\mathrm{i})$ implies $(\mathrm{ii})$ and $(\mathrm{iii})$}], we change $\wh \mu_2 (z)$ to 
\[
 \wh \mu_2(z) =\exp \Big(
\lambda\Lambda_2 \int_{(0,c_1]} (e^{izy}-1)G(dy)/\Lambda_2 
\Big)\quad \text{with}\quad \Lambda_2=G(c_1).
\] 
Then using Proposition \ref{prop:decomp:delta:Poisson} for positive-half distributions, we obtain 
$\mu_1\in \cals_\Delta$. On the way for proving that $\mu_1\in \cals_\Delta$ implies $G_1\in \cals_\Delta$, 
we replace 
Proposition \ref{prop:kesten} with \cite[Theorem 4.25]{Foss:Korshunov:Zachary:2013}. Now 
if Proposition \ref{prop:asympt:equiv} is replaced by Theorem \cite[Theorem 4.22]{Foss:Korshunov:Zachary:2013} 
in the final part, the proof is completed. 
Nothing should be changed for the part [{\bf $(\mathrm{iii})$ implies $(\mathrm{ii})$}]. 
For the part [{\bf $(\mathrm{ii})$ implies $(\mathrm{iii})$ and $(\mathrm{i})$}], it suffices to replace Propositions  
\ref{prop:kesten} and \ref{prop:asympt:equiv} respectively with Theorems 4.25 and 4.22 in \cite{Foss:Korshunov:Zachary:2013}. 
\end{proof}

\begin{proof}[Proof of Theorem \ref{thm:ID:delta:subexponential}]
%Let $\mu \in \id$. % be an infinitely divisible distribution satisfying  (1.3). 
%Suppose that $\nu_{(1)} \in  \call_\Delta $.
We decompose $\mu \in \id$ as
$
 \mu=\mu_1\ast \mu_2,\, \text{i.e.}\, \wh \mu(z)=\wh \mu_1(z)\, \wh \mu_2(z),
$
where 
\[
 \wh \mu_2(z)= \exp \Big(
 \int_{(-\infty,-1)} (e^{izx}-1)\nu(dy)
\Big) \]
%Then $ \mu_2$ 
is the ch.f. of a compound Poisson distribution. Observe that $\mu_2(x+\Delta) =o(e^{-\gamma x})$ and 
$$\int_{\R}e^{-\gamma x}\mu_1(dx) <\infty $$
  for any $\gamma >0$. Suppose that $(\mathrm{i})$ holds.
We obtain from Proposition \ref{prop:decomp:delta:Poisson} that $\mu_1 \in \cals_\Delta \cap \cald_\Delta$ and $\mu(x+\Delta)\sim\mu_1(x+\Delta)$. 
Define the condition  $(\mathrm{iii})'$ as 
\[
(\mathrm{iii})'\qquad  \nu_{(1)} \in \call_\Delta \cap \cald_\Delta\quad \text{and}\quad  \mu_1(x+\Delta)\sim \nu(x+\Delta). 
\]
We see from Theorem 1.2 of \cite{Watanabe:Yamamuro:2009} that the condition $\mu_1 \in \cals_\Delta \cap \cald_\Delta$ is 
equivalent to $(\mathrm{ii})$ and $(\mathrm{iii})'$. Since 
$\mu(x+\Delta)\sim\mu_1(x+\Delta)$ under $\mu_1 \in \cals_\Delta \cap \cald_\Delta$,  $(\mathrm{iii})'$ is equivalent to $(\mathrm{iii})$.
 Conversely, suppose the condition $\mu_1 \in \cals_\Delta \cap \cald_\Delta$ and then $(\mathrm{i})$ follows from  Proposition \ref{prop:asympt:equiv2}. 
\end{proof}

\begin{proof}[Proof of Theorem \ref{thm:ID:delta:subexponential:cp}]
{\bf $[\,(\mathrm{i})$ implies $(\mathrm{ii})$ and $(\mathrm{iii})\,]$}\    
Define $\mu_2 \in \id$ such that $\mu_2$ satisfies $\mu=\mu_{c_1}\ast \mu_2$.
By the same reasoning given as \eqref{deltasub:exp:moment}, 
%Since $\mu_2$ has $e^{\gamma x}$-moment for ant $\gamma>0$, 
%$e^{\gamma x}\mu_2(x+\Delta) \le \int_{x}^{x+c} e^{\gamma y} \mu_2(dy)<\infty$, so that 
$\mu_2(x+\Delta) =O(e^{-\gamma x})$.  
Then by Proposition \ref{prop:factrization:delta} with $(H,F,G)=(\mu,\mu_{c_1},\mu_2)$ we obtain 
$\mu_{c_1}\in \cals_\Delta$. Let $\mu_1$ be a compound Poisson with the L\'evy measure 
$\nu_{(1)}\nu((1,\infty))$ and 
write $\mu_1=\mu_3 \ast \mu_{c_1}$ where $\mu_3$ is a compound Poisson with the L\'evy measure $\nu_3(dx)= {\bf 1}_{\{1<x\le c_1\}}\nu(dx)$. 
%$\nu_3(dx)= {\bf 1}_{\{1<x\le c_1\}}\nu(dx)/\nu((1,c_1])$. 
Similarly as for $\mu_2$ we have 
$\mu_3(x+\Delta)=O(e^{-\gamma x})$ for any $\gamma>0$, so that $\mu_3(x+\Delta)=o(\mu_{c_1}(x+\Delta))$. 
Then by Proposition \ref{prop:asympt:equiv2} with $(F,G_1,G_2)=(\mu_{c_1},\mu_{c_1},\mu_3)$ and $(c_1,c_2)=(1,0)$, 
we obtain $\mu_1\in \cals_\Delta$ and $\mu_1(x+\Delta)\sim \mu_{c_1}(x+\Delta)$. Now since $\mu_1$ is a compound Poisson 
on $\R_+$, due to Theorem \ref{thm:cp:twosided}, $\nu_{(1)} \in \cals_\Delta$ and $\mu_1(x+\Delta) \sim \nu_{(1)}(x+\Delta)\nu((1,\infty))$. 
Thus $(\mathrm{ii})$ follows. In view of $\mu(x+\Delta)\sim \mu_{1}(x+\Delta)$ and 
$\nu(x+\Delta)\sim \nu_{(1)}(x+\Delta)\nu((1,\infty))$, we have already proved $(\mathrm{iii})$. 
Here $\nu_{(1)}\in \cald_\Delta$ follows from $\mu_{c_1} \in \cald_\Delta$ together with the 
asymptotic equivalence $\mu(x+\Delta)\sim \mu_1(x+\Delta)\sim \mu_{c_1}(x+\Delta)$ (cf. Lemma \ref{lem:ald:symp:equiv}).
\\

\noindent
{\bf $[\,(\mathrm{iii})$ implies $(\mathrm{ii})\,]$}\ 
The results immediately follow from Theorem \ref{thm:ID:delta:subexponential}.\\
%Observe that the second condition of $(\mathrm{iii})$ implies $\nu_{(1)}(x+\Delta) \nu((1,\infty))\sim \mu(x+\Delta)$ and thus 
%\begin{align}
%\label{pf:limsup:bound}
% \limsup_{x\to\infty} \frac{\mu_1(x+\Delta)}{\nu_{(1)}(x+\Delta)} \le \limsup_{x\to\infty}
%\frac{\mu_1(x+\Delta)}{\mu(x+\Delta)} \limsup_{x\to\infty} \frac{\mu(x+\Delta)}{\nu_{(1)}(x+\Delta)} \le \nu((1,\infty)),
%\end{align}
%where we use $\mu=\mu_1\ast \mu_2$ and $\mu_1\in \call_\Delta$ (cf. the proof of \cite[Theorem 4.17]{Foss:Korshunov:Zachary:2013}) 
%which implies $\liminf_{x\to\infty}\mu(x+\Delta)/\mu_1(x+\Delta)\ge 1$. For convenience we let $\lambda=\nu((1,\infty))$. 
%Recall that $\mu_1\in \idrp$ is a compound Poisson with the L\'evy measure $\lambda \nu_{(1)}(dx)$ and write 
%\begin{align*}
% \frac{\nu_{(1)}^{\ast 2}(x+\Delta)}{\nu_{(1)}(x+\Delta)} = \frac{2}{\lambda^2} \Big (
%e^{\lambda} \frac{\mu_1(x+\Delta)}{\nu_{(1)}(x+\Delta)}- \sum_{n\neq 2} \frac{\lambda^n}{n!}
%\frac{\nu_{(1)}^{\ast n}(x+\Delta)}{\nu_{(1)}(x+\Delta)}
%\Big ).
%\end{align*}
%Noticing \eqref{pf:limsup:bound}, we take $\limsup_{x\to\infty}$ on both sides. Then Fatou's lemma yields 
%\[
% \limsup_{x\to\infty} \frac{\nu_{(1)}^{\ast 2}(x+\Delta)}{\nu_{(1)}(x+\Delta)} \le 
%\frac{2}{\lambda^2} \Big (
%\lambda e^{\lambda}- \sum_{n\neq 2} \frac{\lambda^n}{n!} \liminf_{x\to\infty}
%\frac{\nu_{(1)}^{\ast n}(x+\Delta)}{\nu_{(1)}(x+\Delta)}\Big ) =2. 
%\]
%Thus by \cite[Corollary 4.19]{Foss:Korshunov:Zachary:2013} with $n=2$, we obtain $\nu_{(1)} \in \cals_\Delta$. \\

\noindent
{\bf $[\,(\mathrm{ii})$ implies $(\mathrm{i})\,]$}\ 
Let $\nu_{(c_1)}(dx)= \nu_{c_1}(dx)/\nu((c_1,\infty))$. 
By Proposition \ref{prop:asympt:equiv} and $\nu_{(1)}(x+\Delta)\nu((1,\infty)) 
\sim \nu_{c_1}(x+\Delta)$, $\nu_{(c_1)}\in \cals_\Delta$ follows. 
Then applying $\nu_{(c_1)}^{\ast n}(x+\Delta)/\nu_{(c_1)}(x+\Delta)\to n$ as $x\to\infty$ together with Kesten bound 
(\cite[Theorem 4.25]{Foss:Korshunov:Zachary:2013}) we obtain by the dominated convergence that 
\begin{align*}
 \lim_{x\to\infty} \frac{\mu_{c_1}(x+\Delta)}{\nu_{(c_1)}(x+\Delta)} =e^{-\nu((c_1,\infty))} \sum_{n=0}^\infty
\frac{(\nu((c_1,\infty)))^n}{n!} \lim_{x \to \infty} \frac{\nu_{(c_1)}^{\ast n}(x+\Delta)}{\nu_{(c_1)}(x+\Delta)} = \nu((c_1,\infty)). 
\end{align*}
Thus $\mu_{c_1} \in \cals_\Delta$. Recall that $\mu_{c_1} \in \cala_\Delta$ by assumption and $\mu_2(x+\Delta)=O(e^{-\gamma x})$ for any $\gamma>0$. 
Now by Proposition \ref{prop:asympt:equiv2} with $(F,G_1,G_2)=(\mu_{c_1},\mu_{c_1},\mu_2)$ and $(c_1=1,c_2=0)$, we obtain $\mu\in \cals_\Delta$.
Finally $\mu\in \cald_\Delta$ follows from $\mu (x+\Delta)\sim \mu_{c_1}(x+\Delta)$ (cf. Lemma \ref{lem:ald:symp:equiv}).  
\end{proof}

\begin{proof}[Proof of Theorem \ref{thm:ID:subexponential:sself}]
{\bf $[\,(\mathrm{ii})$ implies $(\mathrm{i})$ and $(\mathrm{iii})\,]$}\  
Fix $c_1>1$ and we use $\mu_{c_1}$, $\nu_{c_1}$ and $\nu_{(c_1)}$ of Theorem \ref{thm:ID:delta:subexponential:cp}.
Since $\nu_{(1)}(x+\Delta)\nu((1,\infty))\sim \nu_{(c_1)}(x+\Delta)\nu((c_1,\infty))$, 
$\nu_{(c_1)}\in \cals_\Delta$ follows form \cite[Theorem 4.23]{Foss:Korshunov:Zachary:2013}. 
By definition of $s$-self-decomposability we have $\nu_{c_1}\in \cala_\Delta$. 
Recall that $\mu_{c_1}$ is a compound Poisson with the L\'evy measure $\nu_{c_1}$ and thus by Theorem
 \ref{thm:cp:twosided} $(\mathrm{iii})$ $\mu_{c_1}(x+\Delta)\sim \nu_{c_1}(x+\Delta)$, so that 
$\mu_{c_1} \in \cala_\Delta$. Thus Theorem \ref{thm:ID:delta:subexponential:cp} implies $(\mathrm{i})$ and $(\mathrm{iii})$. \\

\noindent
{\bf $[\,(\mathrm{iii})$ implies $(\mathrm{ii})$ and $(\mathrm{i})\,]$}\  
Obviously $\nu_{(1)} \in \cala_\Delta$ and the results follows from Theorem \ref{thm:ID:delta:subexponential}.\\
%By the exactly the same proof for the part [\,$(\mathrm{iii})$ implies $(\mathrm{ii})$\,] of Theorem \ref{thm:cp:twosided} of we obtain 
%$(\mathrm{ii})$, so that $(\mathrm{i})$. \\

\noindent
{\bf $[\,(\mathrm{i})$ implies $(\mathrm{ii})\,]$}\ We prepare three IDs.
Recall that we may write $\nu(dx)=g(x)dx$ since $s$-self-decomposability implies the existence of a L\'evy density. 
Let $\mu_{\ov c_{1}}\in \id$ with 
$a=b=0$ and the L\'evy measure 
\[
 \nu_{\ov c_1}(dx)= ({\bf 1}_{\{x \in [-c_1,c_1]\}}g(c_1)+{\bf 1}_{\{|x|>c_1\}}g(|x|)) dx
\] 
and let $\mu_0 \in \id$ with the same $a$ and $b$ as those of $\mu$ and the L\'evy measure 
\[
 \nu_0(dx)= \big ({\bf 1}_{\{x \in [-c_1,c_1],\,g(x)-g(c_1)\ge 0\}(g(x)-g(c_1))}+{\bf 1}_{\{x<-c_1\}}g(x) \big) dx
\]
and let $\mu_- \in \id$ with $a=b=0$ and the L\'evy measure 
\[
\nu_-(dx)=\big ( {\bf 1}_{\{x<-c_1\}}g(-x)
+ {\bf 1}_{\{x \in [-c_1,0],\,g(c_1)-g(x)\ge 0\}}(g(c_1)-g(x))
\big)  dx. 
\]
%$\nu_-(dx)={\bf 1}_{\{x<-c_1\}}g(-x)dx $. 
Here observe that 
$\mu_-(x+\Delta)=O(e^{-\gamma x})$ and $\mu_0(x+\Delta)=O(e^{-\gamma x})$ for any $\gamma>0$. 
Then $\mu \ast \mu_- \in \cals_\Delta \cap \cald_\Delta$ follows from Proposition \ref{prop:asympt:equiv2} with 
$(F,G_1,G_2)=(\mu,\mu,\mu_-)$ and $(c_1=1,c_2=0)$ (cf. Lemma \ref{lem:ald:symp:equiv}). Notice that 
$\nu_{\ov c_1}$ is symmetric and unimodal, so is the compound Poisson $\mu_{\ov c_1}$ and $\mu_{\ov c_1}\in \cala_\Delta$.
Since $\mu_{\ov c_1}\ast \mu_0= \mu\ast \mu_- \in \cals_\Delta \cap \cald_\Delta$ 
by Proposition \ref{prop:factrization:delta}, $\mu_{\ov c_1} \in \cals_\Delta \cap \cala_\Delta$. 
Recall that $\mu_{\ov c_1}$ is a compound Poisson with L\'evy measure 
$\nu_{\ov c_1}$. Write $\nu_{(\ov c_1)}(dx)=\nu_{\ov c_1}(dx)/(2(c_1 g(c_1)+G(c_1,\infty)))$ and then due to 
Corollary \ref{col:cp:twosided:ani}, $\nu_{(\ov c_1)} \in \cals_\Delta \cap \cala_\Delta$ follows. %, so that $\nu_{\ov c_1}\in \cals_\Delta \cap \cala_\Delta$. 
Now since $\nu_{\ov c_1}(x+\Delta) \sim \nu_{(1)}(x+\Delta) \nu((1,\infty))$, we have 
$\nu_{(1)} \in \cals_\Delta$.
\end{proof}

\begin{proof}[Proof of Theorem \ref{thm:ID:loc:subexponential}]
%Let $\mu$ be an infinitely divisible distribution satisfying  (1.3).
We decompose $\mu$ as in the proof of Theorem  \ref{thm:ID:delta:subexponential}.  Suppose that $(\mathrm{i})$ holds.
We obtain from Proposition \ref{prop:decomp:delta:Poisson} that $\mu_1 \in \cals_{loc} \cap \cald_{loc}$ and 
$\mu(x+\Delta)\sim\mu_1(x+\Delta)$ for any $\Delta$. %Define the condition $(\mathrm{iii})'$ as 
Introduce the condition 
\[
 (\mathrm{iii})'\qquad \nu_{(1)} \in \call_{loc} \cap \cald_{loc}\quad \text{and}\quad \mu_1(x+\Delta)\sim \nu(x+\Delta) 
\quad \text{for any}\ \Delta.  
\]
% (iii)' $\nu_{(1)} \in \call_{loc} \cap \cald_{loc}$ and $ \mu_1(x+\Delta)\sim \nu(x+\Delta)$ for any $\Delta$. 
We see from Corollary 2.1 of \cite{Watanabe:Yamamuro:2010} that the condition $\mu_1 \in \cals_{loc} \cap \cald_{loc}$ is 
equivalent to $(\mathrm{ii})$ and $(\mathrm{iii})'$. Since 
$\mu(x+\Delta)\sim\mu_1(x+\Delta)$ for any $\Delta$, $(\mathrm{iii})'$ is equivalent to $(\mathrm{iii})$. 
Conversely, suppose the condition $\mu_1 \in \cals_{loc} \cap \cald_{loc}$. Then $(\mathrm{i})$ follows from
 Proposition \ref{prop:asympt:equiv2}.
\end{proof}

\begin{proof}[Proof of Theorem \ref{theoremcompound}]  
We give the proof only for the general two-sided case, since that for 
the positive-half case is similar. As in the proof of Theorem  \ref{thm:cp:twosided}  
we can prove that $(\mathrm{iii})$ implies $(\mathrm{ii})$ and that $(\mathrm{ii})$ implies $(\mathrm{iii})$ and $(\mathrm{i})$. 
Thus it is enough to prove that 
$(\mathrm{i})$ implies $(\mathrm{ii})$. Define $\delta:=-\log(1-\lambda)$ and the distribution $\nu_0$ as
$$\nu_0:= \delta^{-1} \sum_{n=1}^{\infty}(\lambda^n/ n)\,G^{*n}.$$
Then, as shown in the proof of \cite[Theorem 1.8]{Watanabe:2008} (cf. the proof of \cite[Corollary 3]{Embrechts:Goldie:Veraverbeke:1979})
that 
$$F_a =e^{-\delta a}\sum_{n=0}^{\infty}( (\delta a)^n/ n!)\,\nu_0^{*n}\quad 
\text{and}\quad 
 G=-\lambda^{-1}\sum_{n=1}^{\infty}\big( (-\delta)^n /n!\big)\,\nu_0^{*n}.$$
Suppose that $F_a \in  \cals_\Delta \cap \cald_\Delta$ for some $a>0$, or equivalently for all $a>0$. %(see Theorem  \ref{thm:cp:twosided}). 
Then we obtain from Theorem  \ref{thm:cp:twosided} that $\nu_0 \in  \cals_\Delta \cap \cald_\Delta$. 
Now due to Propositions \ref{prop:asympt:equiv3} and \ref{prop:kesten},  
the dominated convergence theorem yields 
%Thus, by using the dominated convergence theorem, we find from Propositions \ref{prop:asympt:equiv3}  and \ref{prop:kesten} that
$$\lim_{x \to \infty}\frac{G(x +\Delta)}{\nu_0(x +\Delta)} =\delta e^{-\delta}/\lambda >0$$
and hence by Proposition \ref{prop:asympt:equiv} that $G\in  \cals_\Delta \cap \cald_\Delta$. 
\end{proof}

\begin{proof}[Proof of Theorem \ref{theoremsupremum}]   Let $F_a=\pi,$ $G=G_0$, and $a=1$ in Theorem \ref{theoremcompound}. 
Then the proof is clear from Theorems \ref{thm:cp:twosided} and \ref{theoremcompound}.
\end{proof}

\end{document}